\documentclass[a4paper,12pt%
]{amsart}

\usepackage{amsmath, amssymb, amsthm, amscd, tikz}
\usepackage{graphicx}
\usepackage{enumerate}
\usepackage{caption}
\usepackage{subcaption}
\usepackage{comment}
\usepackage[all,cmtip]{xy}

\usepackage{geometry}

\usepackage{a4wide}

\usetikzlibrary{arrows}

\usepackage{algorithm}
\usepackage{algpseudocode}

\usepackage[english]{babel}
\usepackage[latin1]{inputenc}
\usepackage[all,cmtip]{xy}
\usepackage{times}
\usepackage{cancel}

\newcounter{theorem}
\newtheorem{thm}[theorem]{Theorem}
\newtheorem{lemma}[theorem]{Lemma}
\newtheorem{prop}[theorem]{Proposition}
\newtheorem{cor}[theorem]{Corollary}
\newtheorem{defn}[theorem]{Definition}

\newtheorem{question}[theorem]{Question}

\theoremstyle{remark}
\newtheorem*{remark*}{Remark}
\newtheorem{remark}[theorem]{Remark}
\newtheorem{example}[theorem]{Example}

\numberwithin{equation}{section}
\numberwithin{theorem}{section}

\newcommand{\defemph}{\emph}

\newcommand{\R}{\mathbb{R}}

\newcommand{\Z}{\mathbb{Z}}
\newcommand{\N}{\mathbb{N}}

\newcommand{\sub}{\varphi} 
\newcommand{\TS}{\Omega} 
\newcommand{\Al}{\mathcal{A}} 
\newcommand{\Be}{\mathcal{B}} 
\newcommand{\Ga}{\mathcal{C}} 
\newcommand{\Ze}{\mathcal{Z}} 

\renewcommand{\setminus}{\backslash}
\renewcommand{\emptyset}{\varnothing}
\newcommand{\Cech}{\v{C}ech{} }

\title[Beyond primitivity]{Beyond primitivity for one-dimensional substitution subshifts and tiling spaces}

\author{Gregory R. Maloney}
\address{Newcastle University}
\thanks{}

\author{Dan Rust}
\address{Bielefeld University}
\thanks{}

\begin{document}

\date{\today}
\thanks{The authors were partly supported by Grant IN-2013-045 from the Leverhulme Trust for International Network.}

\begin{abstract}
We study the topology and dynamics of subshifts and tiling spaces associated to non-primitive substitutions in one dimension.  We identify a property of a substitution, which we call tameness, in the presence of which most of the possible pathological behaviours of non-minimal substitutions cannot occur.  We find a characterisation of tameness, and use this to prove a slightly stronger version of a result of Durand, which says that the subshift of a minimal substitution is topologically conjugate to the subshift of a primitive substitution.  

We then extend to the non-minimal setting a result obtained by Anderson and Putnam for primitive substitutions, which says that a substitution tiling space is homeomorphic to an inverse limit of a certain CW complex under a self-map induced by the substitution.  We use this result to explore the structure of the lattice of closed invariant subspaces and quotients of a substitution tiling space, for which we compute cohomological invariants that are stronger that the \Cech cohomology of the tiling space alone.  
\end{abstract}

\maketitle	
	
	\section{Preliminaries}\label{SEC:prelims}
	\subsection{Outline}\label{SEC:outline}
	The goal of this work is to study one-dimensional tiling spaces arising from non-primitive substitution rules, in terms of the topology, dynamics, and cohomology.
This study naturally divides into two cases: the case where the tiling space is minimal, and the case where it is non-minimal.  
The minimal case is treated in Section \ref{SEC:minimal}, where we identify a property of a substitution, which we call \defemph{tameness}, in the presence of which most of the possible pathological behaviours of non-minimal substitutions cannot occur. In particular, all aperiodic substitutions will be seen to be tame.  By aperiodic, we mean that the subshift of the substitution has no periodic orbits.
The first main result is Theorem \ref{THM:wild-characterisation}, which gives a characterisation of tameness.  
This theorem is used to prove the following result.  
\newtheorem*{THM:main}{Theorem \ref{THM:main}}
\begin{THM:main}
Let $\sub$ be a minimal substitution with non-empty minimal subshift $X_\sub$.  
There exists an alphabet $\Ze$ and a primitive substitution $\theta$ on $\Ze$ such that $X_\theta$ is topologically conjugate to $X_\sub$.
\end{THM:main}

This is similar to, but slightly stronger than, a result from the section on Open problems and perspectives (Section 6.2) of \cite{D:main-result}. 

Examples and applications based on this section are then given in Section \ref{SEC:examples}. 

The non-minimal case is treated in Sections \ref{SEC:non-minimal}, \ref{SEC:CISs} and \ref{SEC:non-min-examples}.  
Non-minimal substitution tiling spaces fall outside the scope of much of the existing literature, which typically deals only with primitive systems.  Nevertheless a theory of non-minimal substitutions is important, even to the study of primitive substitutions, because there are invariants of primitive substitution tiling spaces that are themselves non-minimal substitution tiling spaces.  These invariants originally appeared in \cite{BD:proximal}, and are discussed in more detail in Examples \ref{EX:proximal} and \ref{EX:asymptotic}.  

The main result of Section \ref{SEC:non-minimal} is the following theorem, which is an extension to the non-minimal setting of results obtained by Anderson and Putnam for primitive substitutions \cite{AP}, although their results apply in arbitrary dimension.  
\theoremstyle{theorem}	
\newtheorem*{THM:main2}{Theorem \ref{THM:homeo}}
\begin{THM:main2}
Let $\sub$ be a tame recognisable substitution. There exists a complex $\Gamma$ and a map $f\colon \Gamma \to \Gamma$ such that there is a homeomorphism $h \colon \TS_\sub \to \varprojlim(\Gamma,f)$.
\end{THM:main2}

Section \ref{SEC:CISs} is devoted to building a structure theorem of non-minimal tiling spaces in terms of their closed shift-invariant subspaces. In particular, we identify a correspondence between such subspaces and subcomplexes of the complex $\Gamma$ above. The subspaces are found to be homeomorphic to an inverse limit of self-maps acting on the corresponding subcomplex of $\Gamma$.

Examples are given in Section \ref{SEC:non-min-examples} to justify the level of care that needs to be taken in building the machinery, and to give an exposition of how the machinery is put to use when performing calculations.

\subsection{Subshifts and Tiling Spaces}\label{SUBSEC:subshifts}

		Let $\Al$ be a finite alphabet and for natural numbers $n$, let $\Al^n$ be the set of words of length $n$ using symbols from $\Al$. We denote the length of the word $u = u_1\ldots u_l$ by $|u| = l$. By convention, $\Al^0 = \{\epsilon\}$ where $\epsilon$ is the empty word and $|\epsilon|=0$. Denote the union of the positive-length words by $\Al^+ = \bigcup_{n \geq 1} \Al^n$. If the empty word $\epsilon$ is also included, then we denote the union $\Al^+ \cup \{\epsilon\}$ by $\Al^\ast$. This set $\Al^*$ forms a free monoid under concatenation of words.

		A \emph{substitution} $\sub$ on $\Al$ is a function $\sub \colon \Al \to \Al^+$. We can extend the substitution $\sub$ in a natural way to a morphism $\sub \colon \Al^* \to \Al^*$ given, for a word $u = u_1\ldots u_n \in \Al^n$, by setting $\sub(u) = \sub(u_1) \ldots \sub(u_n)$. The symbol $w_i$ denotes the label assigned to the $i$th component of the bi-infinite sequence $w \in \Al^\Z$. We may further extend the definition of a substitution to bi-infinite sequences $\sub \colon \Al^\Z \to \Al^\Z$. For a bi-infinite sequence $w \in \Al^{\Z}$, with $w = \ldots w_{-2} w_{-1} \cdot w_0 w_1 w_2 \ldots$ we set $$\sub(w) = \ldots \sub(w_{-2}) \sub(w_{-1}) \cdot \sub(w_0) \sub(w_1) \sub(w_2) \ldots$$ with the dot $\cdot$ representing the separator of the $(-1)$st and $0$th component of the respective sequences.
		
		For a substitution $\sub \colon \Al \to \Al^+$ on an alphabet $\Al = \{a^1, a^2, \ldots, a^l\}$, there is an associated substitution matrix $M_{\sub}$ of dimension $l\times l$ given by setting $m_{ij}$, the $i,j$ entry of $M_{\sub}$, to be the number of times that the letter $a^i$ appears in the word $\sub(a^j)$.

		A substitution $\sub$ is called \emph{primitive} if there exists a positive natural number $p$ such that the matrix $M_{\sub}^p$ has strictly positive entries. Equivalently, if there exists a positive natural number $p$ such that for all $a,a'\in\Al$ the letter $a'$ appears in the word $\sub^p(a)$.

		For words $u, v \in \Al^*$, we write $u \subset v$ to mean $u$ is a subword of $v$, and $u \subsetneq v$ to mean $u$ is a proper subword of $v$. For a bi-infinite word $w \in \Al^\Z$, we similarly write $u \subset w$ to mean $u$ is a subword of $w$.

			Let $\sub \colon \Al \to \Al^+$ be a substitution. We say a word $u \in \Al^\ast$ is \emph{admitted} by the substitution $\sub$ if there exists a letter $a\in\Al$ and a natural number $k\geq 0$ such that $u \subset \sub^k(a)$ and denote by $\mathcal{L}^n \subset \Al^n$ the set of all words of length $n$ which are admitted by $\sub$. Our convention is that the empty word $\epsilon$ is admitted by all substitutions. We form the \emph{language} of $\sub$ by taking the set of all admitted words $\mathcal{L} = \bigcup_{n\geq 0} \mathcal{L}^n$.

			We say a bi-infinite sequence $w \in \Al^\Z$ is \emph{admitted} by $\sub$ if every subword of $w$ is admitted by $\sub$ and denote by $X_\sub$ the set of all bi-infinite sequences admitted by $\sub$. The set $X_\sub$ has a natural (metric) topology inherited from the product topology on $\Al^\Z$ and a natural shift map $\sigma \colon X_\sub \to X_\sub$ given by $\sigma(w)_i= w_{i+1}$. We call the pair $(X_\sub, \sigma)$ the \emph{subshift} associated to $\sub$ and we will often abbreviate the pair to just $X_{\sub}$ when the context is clear.

			We remark that it is not necessarily the case that every word in the language of a substitution appears as the subword of a sequence in the subshift - for example $ab$ is in the language of the substitution $\sub \colon a \mapsto ab,\: b \mapsto b$, but the subshift for this substitution is the single periodic sequence $\ldots bbb \ldots$ which does not contain $ab$ as a subword.
			
			We say a word $u$ is \emph{legal} if it appears as a subword of a sequence of the subshift for the substitution $\sub$. Then the set $\hat{\mathcal{L}}_\sub$ of legal words for $\sub$ is a subset of the language $\mathcal{L}_\sub$. If $\hat{\mathcal{L}}_\sub = \mathcal{L}_\sub$ then we say that $\sub$ is an \emph{admissible} substitution---this definition is borrowed from \cite{CS:non-primitive-invariant-measures} where they give the equivalent definition that $\sub$ is admissible if every $a \in \Al$ is a legal letter. Every primitive substitution is admissible.
	Some of the results of this work would be simplified if we chose to focus only on admissible substitutions, however this will not be an assumption that we make.
	
	Let $L$ be a non-empty subset of the subshift $X_\sub$.  
If, for every point $w$ in $L$, it is true that $L = \overline {\{\sigma^i(w)\}}_{i \in \Z}$, the orbit closure of $w$, then $L$ is called a \emph{minimal component} of $X_\sub$.  
If the subshift $X_\sub$ is a minimal component of itself, then $\sub$ is called a minimal substitution and $X_\sub$ is called a \emph{minimal} subshift, otherwise $\sub$ and $X_\sub$ are called \emph{non-minimal}.

For a primitive substitution, any admitted word is also legal.
If $u$ and $v$ are words, let us use the notation $|v|_u$ to denote the number of occurrences of $u$ as a subword of $v$.  
A subshift is called \defemph{linearly recurrent} if there exists a natural number $C \in \N$ such that, for all legal words $u$ and $v$, if $|v| > C|u|$, then $|v|_u \geq 1$.  

One fact that will play an important role in this section is the following, which was proved in \cite{DL:linearly-recurrent}.
\begin{thm}[Damanik--Lenz]\label{THM:linearly-recurrent}
Let $\sub$ be a substitution on $\Al$.  
The subshift $X_\sub$ is minimal if and only if it is linearly recurrent.
\end{thm}

			We say $\sub$ is a \emph{periodic} substitution if $X_\sub$ is finite, and $\sub$ is \emph{non-periodic} otherwise. We say $\sub$ is \emph{aperiodic} if $X_\sub$ contains no $\sigma$-periodic points (equivalently, $X_\sub$ contains no periodic closed invariant subspaces). If $\sub$ is non-periodic and primitive, then $X_\sub$ is aperiodic and topologically a Cantor set (in particular $X_\sub$ is non-empty) and $\sub$ a minimal substitution.
			\begin{remark}
			Often in the literature, the subshift of a substitution is defined in terms of the $\Z$-orbit closure of a particular bi-infinite sequence. For primitive substitutions this sequence-dependent definition coincides with our definition in terms of a language. The terms periodic, non-periodic and aperiodic are then normally associated to properties of a sequence, rather than its subshift. For our purposes, we need to define a substitution subshift independent of some generating sequence, as not all substitution subshifts contain a dense orbit.
			
			One should therefore be careful when justifying the use of the term \emph{non-periodic} to refer to a subshift which may nevertheless contain periodic points under the shift action. In particular, there exist non-periodic substitutions which contain sequences of all three types with respect to the sequence-based definition. For example $\sub\colon a\mapsto aa, b\mapsto aba, c\mapsto ccd, d\mapsto cd, e\mapsto bdecb$.
			
			This potentially confusing naming convention should then be considered in this context: non-periodic substitution subshifts \emph{contain} a point with an infinite orbit (but may also contain finite orbits); aperiodic substitution subshifts consist \emph{only} of points with infinite orbits.
			\end{remark}

		\subsection{Tiling Spaces}
			Let $\sub$ be a substitution on the alphabet $\Al$ with associated subshift $X_\sub$. The \emph{tiling space} associated to $\sub$ is the quotient space $$\TS_\sub = (X_\sub \times [0,1]) /{\sim}$$ where $\sim$ is generated by the relation $(w,0)\sim (\sigma(w),1)$.

The natural translation action $T \mapsto T+t$ for $t\in\R$ equips $\TS_\sub$ with a continuous $\R$ action which is minimal whenever $\sub$ is primitive. In this respect, tiling spaces are closely related to the more well-known spaces, the solenoids. To some degree, tiling spaces may be thought of informally as non-homogeneous solenoids. We note that there exist non-primitive substitutions with associated tiling spaces whose translation action is minimal, so primitivity is only a sufficient condition for minimality. This will be explored in Section \ref{SEC:minimal}.
		
		\begin{defn}
		Let $w = \ldots w_{-2} w_{-1} \cdot w_0 w_1 w_2 \ldots$ be a bi-infinite sequence in $X_\sub$ and let $t\in[0,1)$, so that $(w,t)$ is an element of the tiling space $\TS_\sub$. We define a map on the tiling spaces which we call $\sub \colon \TS_\sub \to \TS_\sub$, given by $$\sub(w,t) = (\sigma^{\lfloor \tilde{t} \rfloor}(\sub(w)), \tilde{t}-\lfloor \tilde{t} \rfloor)$$ where $\tilde{t} = |\sub(w_0)| \cdot t$ and $\lfloor - \rfloor$ is the floor function.
		\end{defn}
		
		This map is continuous. Intuitively, we take a unit tiling in $\TS_\sub$ with a prescribed origin and partition each tile of type $a$ uniformly with respect to the substituted word $\sub(a)$ into tiles of length $\frac{1}{|\sub(a)|}$. We then expand each tile away from the origin so that each new tile is again of unit length, and with the origin lying proportionally above the tile it appears in after partitioning the original tiling.

		A substitution $\sub$ is said to be \emph{recognisable} if the map $\sub \colon \TS_\sub \to \TS_\sub$ is injective.

		It is a result of Moss\'{e} \cite{M:aperiodic} that a primitive substitution is aperiodic if and only if it is recognisable.
		
As with subshifts, there is a notion of minimality and minimal components for tiling spaces.  
We call $\Lambda \subset \TS_\sub$ a \defemph{minimal component} of $\TS_\sub$ if $\Lambda = (L\times I)/\sim$ for some minimal component $L$ of the subshift $X_\sub$, and we say that $\TS_\sub$ is a \defemph{minimal tiling space} if it is a minimal component of itself. In Section \ref{SEC:non-minimal} this notion of minimality will be extended to any compact dynamical system, but for now this definition is more convenient.

There are many properties of primitive substitutions which one is likely to take for granted, and so we take this opportunity to explicitly spell out some of these properties and how such properties can fail in the general case (giving both minimal and non-minimal examples where appropriate).

The following results can be found in various places in the literature. We refer the reader to \cite{S:book} for a concise resource of proofs for most of these results.

\begin{prop}
	Let $\Al$ be an alphabet on $k$ letters. If $\sub \colon \Al \to \Al^+$ is primitive, then:
	\begin{enumerate}
		\item $X_\sub$ is non-empty
		\item $X_{\sub^n} = X_\sub$ for all $n \geq 1$
		\item $|\sub^n(a)| \to \infty$ as $n \to \infty$ for all $a \in \Al$
		\item $\sigma \colon X_\sub \to X_\sub$ is minimal. In particular $\TS_\sub$ is connected
		\item $\sub$ is non-periodic if and only if $\sub$ is aperiodic
		\item $\operatorname{rk}\check{H}^1(\TS) \leq k^2-k+1$ (see \cite{GM:multi-one-d} and \cite{R:mixed-subs})
		\item $\TS_\sub$ has at most $k^2$ asymptotic orbits (see \cite{BDH:asymp-orbits})
		\item If $\sub$ is recognisable then $\sub$ is non-periodic
	\end{enumerate}
\end{prop}

\begin{prop}Counter examples to the above listed properties in the absence of primitivity are given by:
	\begin{enumerate}
		\item Let $\Al = \{a, b\}$. If $\sub \colon a\mapsto b,\: b\mapsto a$ then $X_\sub$ is empty
		\item Let $\Al = \{0, \overline{0}, 1, X\}$. If $\sub \colon 0 \mapsto \overline{0}\overline{0}1\overline{0},\: \overline{0} \mapsto 0010,\: 1\mapsto 1,\: X \mapsto 0\overline{0}$ then $0\overline{0} \in \hat{\mathcal{L}}_\sub$ but $0\overline{0} \notin \hat{\mathcal{L}}_{\sub^2}$ and so $X_{\sub^2}\subsetneq X_\sub$
		\item Let $\Al = \{a, b, c\}$. If $\sub \colon a\mapsto aaca,\: b\mapsto b,\: c\mapsto bb$ then $|\sub^n(b)| \to 1$ and $|\sub^n(c)| \to 2$ as $n \to \infty$. For a non-minimal case, see the above example for point 2 and the letter $1$
		\item See the counterexample for point 2 for a connected example. The substitution $a \mapsto ab,\: b \mapsto a,\: c\mapsto cd,\: d \mapsto c$ has a tiling space $\TS$ with two connected components
		\item Let $\Al = \{a, b, c, d\}$. If $\sub \colon a \mapsto ab,\: b\mapsto a,\: c \mapsto cc,\: d \mapsto ca$ then $$X_\sub = X_{Fib} \sqcup \bigcup_{n \in \Z} \{\sigma^n(\ldots ccc.abaab\ldots)\} \sqcup \{\ldots cc.cc \ldots\}$$ where $Fib$ is the Fibonacci substitution given by restricting $\sub$ to the subalphabet $\{a, b\}$. The substitution $\sub$ is not aperiodic because it contains the point $\ldots cc.cc \ldots$ which is fixed under $\sigma$. The substitution $\sub$ is non-periodic because $X_{Fib}$ is infinite
		\item A minimal counterexample will be given in Section \ref{SEC:examples}
		\item A minimal counterexample will be given in Section \ref{SEC:examples}
		\item Let $\Al = \{a, b\}$. If $a \mapsto ab,\: b\mapsto b$ then $X_\sub = \{\ldots bb.bb \ldots\}$ and $\TS_\sub$ is homeomorphic to a circle, with the induced substitution map $\sub \colon \TS_\sub \to \TS_\sub$ acting as the identity, hence is injective. It follows that $\sub$ is recognisable, but not non-periodic
	\end{enumerate}
\end{prop}


\section{The Minimal Case}\label{SEC:minimal}

There are two main results in this section.  
Firstly, we identify a property, which we call \defemph{tameness}, such that any substitution with this property is well behaved, in the sense that it does not exhibit certain pathologies that can occur in the non-primitive setting.  
Then the first main result is Theorem \ref{THM:wild-characterisation}, which characterises tameness.  

The second main result of this section is the following.  

\begin{thm}\label{THM:main}
Let $\sub$ be a minimal substitution with non-empty minimal subshift $X_\sub$.  
There exists an alphabet $\Ze$ and a primitive substitution $\theta$ on $\Ze$ such that $X_\theta$ is topologically conjugate to $X_\sub$.
\end{thm}

The idea of the theorem is that non-primitive substitutions are `pathological' and primitive ones are `well behaved', and the theorem makes it possible to replace a non-primitive substitution with a primitive one if the substitution is minimal.  
This is similar to, but slightly stronger than, a result from the section on Open problems and perspectives (Section 6.2) of \cite{D:main-result}.  
There are three reasons for presenting this result here.  
Firstly, the result of \cite{D:main-result} does not appear to be well known, but is basic enough that it seems worthwhile to draw attention to it.  
Secondly, the proof appearing in \cite{D:main-result} is only a sketch, using the Chacon substitution as an illustrative example, whereas a complete proof appears here.  
Thirdly, the result presented here applies to the bi-infinite context, whereas the result of \cite{D:main-result} applies to the one-sided infinite context.  
In this bi-infinite context there are significant pathologies that are not possible in the one-sided infinite context.  
In particular, there exist examples (such as Example \ref{EX:periodic}) of minimal substitutions for which the subshift is not generated by any finite seed.  
Corollary \ref{COR:aperiodic_implies_tame}, which was originally proved in \cite{BKM:recognisable}, and which is a consequence of Theorem \ref{THM:wild-characterisation}, implies that such a pathological substitution must give rise to a periodic subshift, in which case it is easy to find a primitive substitution giving rise to the same subshift.  

\begin{example}\label{EX:periodic}
Let 
\[
\sub \colon \left\{ \begin{array}{l} a \mapsto ab \\ b \mapsto b \end{array} \right..
\]
Then $X_\sub$ is periodic---it contains only the constant sequence $\ldots bbb\ldots$.  
This sequence does not contain any instance of the letter $a$, which is the only letter of which the images under $\sub^n$ grow without bound.  
\end{example}

\subsection{Periodicity and Aperiodicity}\label{SUBSEC:periodicity}

The lemmas below divide the class of substitutions with minimal subshifts into two subclasses, depending upon whether or not there is any legal letter, the length of which grows without bound under $\sub$.  
In particular, Corollary \ref{COR:aperiodic_implies_tame} implies that, in the absence of such a legal letter, a minimal subshift must be periodic, as in Example \ref{EX:periodic}.  
Some of these results apply more generally to non-minimal substitutions, and will see further use in Section \ref{SEC:non-minimal}.  

These lemmas involve a partition of the alphabet into two subsets.  
\begin{defn}\label{DEF:bounded-expanding}
For a substitution $\sub \colon \Al \to \Al^+$, let us say that a word $u\in \Al^+$ is \defemph{bounded} with respect to $\sub$ if there exists $M\in\N$ such that $|\sub^n(u)|\leq M$ for all $n\in\N$, and \defemph{expanding} if it is not bounded.  

Let us denote the set of bounded letters for $\sub$ by $\Al_B$, and the set of expanding letters by $\Al_\infty$.  
\end{defn}

Then $\Al$ is the disjoint union of $\Al_B$ and $\Al_\infty$, and if $X_\sub$ is non-empty then $\Al_\infty$ is non-empty.  
Note also that, for every $b \in \Al_\infty$, $\sub(b)$ must contain at least one letter in $\Al_\infty$, whereas for every $a \in \Al_B$, $\sub(a)$ contains only letters in $\Al_B$.  

The following definitions will be useful in separating the badly behaved non-primitive substitutions from the well behaved ones.  
	\begin{defn}\label{DEF:wild}
	Let $\sub$ be a substitution, and let $B$ denote the set of bounded legal words for $\sub$. If $B$ is finite, we say $\sub$ is \emph{tame}. If $B$ is infinite, we say $\sub$ is \emph{wild}.
	\end{defn}

The substitution in Example \ref{EX:periodic} is wild, as the periodic sequence $\ldots bb.bb \ldots$ is an element of the subshift and the words $b^n$ are all bounded for $\sub$.

	\begin{example}
		The substitution $\sub' \colon a \mapsto abb,\: b \mapsto bbb$ is tame, as $|\sub(u)| = 3|u|$ for all words $u$.
	\end{example}
	 Note that the subshifts $X_\sub$ and $X_{\sub'}$ for these two examples are the same, so tameness is only a property of a substitution and not its associated subshift.  

A particular goal of introducing these definitions is to to show that, if a minimal substitution is wild, then it is periodic.  

Wildness can be characterised in terms of the following two sets and the property $(\ast)$.  
\begin{defn}\label{DEF:right-bounded}
	Let $\Al_{right}\subset \Al_\infty$ denote the set of expanding letters such that for every $a\in\Al_{right}$ the rightmost letter of $\sub(a)$ is a bounded letter, and define $\Al_{left}$ similarly.  
\end{defn}
\begin{defn}
	Suppose that either there exists a letter $a\in \Al_{right}$ and an increasing sequence of integers $N_i$ such that the rightmost expanding letter appearing in $\sub^{N_i}(a)$ is also in $\Al_{right}$ for all $i\geq 1$, or else there exists a letter $a\in \Al_{left}$ and an increasing sequence of integers $N_i$ such that the leftmost expanding letter appearing in $\sub^{N_i}(a)$ is also in $\Al_{left}$ for all $i\geq 1$.
	
	In such a case, we say that $\sub$ has property $(\ast)$.
\end{defn}

The following lemma is used in the proof of Theorem \ref{THM:wild-characterisation}, which characterises wildness.  
        \begin{lemma}\label{LEM:leftmost-rightmost-periodic}
        Let $\sub$ be a substitution on $\Al$ with property $(\ast)$.  Then $X_\sub$ contains a periodic sequence, the letters of which are all bounded.
        \end{lemma}

	\begin{proof}
	Suppose that there exists a letter $a\in \Al_{right}$ and an increasing sequence of integers $N_i$ such that the rightmost expanding letter of $\sub^{N_i}(a)$ is also in $\Al_{right}$; the case with leftmost expanding letters is similar.  Note that the rightmost expanding letter of $\sub^{N_k}(a)$ must also have the same property as $a$ for the shifted sequence of integers $M_{i}= N_{i-k}$. So, by possibly choosing a different $a\in \Al_{right}$ we may further assume without loss of generality that there is a power $N$ so that the rightmost expanding letter of $\sub^{N}(a)$ is $a$. So, let $\sub^{N}(a) = vau$ where $u$ is a bounded word. Then by induction, we have $$\sub^{(k+1)N}(a) = \sub^{kN}(v)\ldots \sub^{N}(v) v a u \sub^N(u) \ldots \sub^{kN}(u).$$
	
	Now, as $u$ is a bounded word, there exists a $K$ such that $|\sub^{(K+1)N}(u)| = |\sub^{KN}(u)|$ and as there are only finitely many words of this length, by possibly replacing $\sub$ with a power, we can choose $K$ such that $\sub^{(K+1)N}(u) = \sub^{KN}(u)$. So for all $j \geq K$, the word $(\sub^{KN}(u))^j$ appears as a subword of $\sub^n(a)$ for some $n$. As such, the periodic sequence $$\ldots \sub^{KN}(u) \sub^{KN}(u)\sub^{KN}(u) \ldots$$ is admitted by $\sub$. This means that the subshift $X_\sub$ contains a periodic point, and, as $u$ is bounded, so is $\sub^{KN}(u)$.
	\end{proof}

The main result of this section is the following theorem, which characterises wildness.  
	\begin{thm}\label{THM:wild-characterisation}
	Let $\sub$ be a substitution on an alphabet $\Al$.  
The following are equivalent:  
\begin{enumerate}
\item  $\sub$ is wild.  
\item  $\sub$ has property $(\ast)$.
\end{enumerate}
	\end{thm}
\begin{proof}
The fact that (2) implies (1) is an immediate consequence of Lemma \ref{LEM:leftmost-rightmost-periodic}.  

To see that (1) implies (2), suppose that $\sub$ does not have property $(\ast)$. Then there is no $a\in \Al_{left}$ and increasing sequence of integers $N_i$ such that the leftmost expanding letter appearing in $\sub^{N_i}(a)$ is also in $\Al_{left}$ for all $i\geq 1$. By this assumption, there exists an $N$ such that the leftmost expanding letter in $\sub^{N+k}(a)$ is never in $\Al_{left}$ for any expanding letter $a$. Let $U_{left}$ be the set of bounded words that appear at the start of any word of the form $\sub^n(a)$ for any expanding letter $a$. This set is finite because $\sub^{N+k}(a)$ will be a word of the form $ubv$ where $u$ is bounded and the leftmost letter of $\sub(b)$ is expanding and also not in $\Al_{left}$. Let $k_{left} = \max\{|u| \mid u\in U_{left}\}$.
	
	As $\sub$ does not have property $(\ast)$, there is also no $a\in \Al_{right}$ and increasing sequence of integers $N_i$ such that the leftmost expanding letter appearing in $\sub^{N_i}(a)$ is also in $\Al_{right}$ for all $i\geq 1$. Then we can similarly form $U_{right}$, the set of bounded words that appear at the end of any word of the form $\sub^n(a)$ for $a\in\Al_{right}$. Let $k_{right} = \max\{|u| \mid u\in U_{right}\}$.
	
	It is easy to see that the only legal bounded words for $\sub$ are either bounded words appearing as subwords contained in the interior of $\sub(a)$ for an expanding $a$, or words of the form $u_1u_2$ for $u_1\in U_{right}$ and $u_2\in U_{left}$. It follows that every bounded word has length at most $\max\{ k_{left} + k_{right}, |\sub(a)| \mid a \in \Al\}$ and so $\sub$ is tame.
\end{proof}
	\begin{cor}\label{COR:leftmost-rightmost1}
	 If the leftmost and rightmost letter of $\sub(a)$ are elements of $\Al_\infty$ for all $a\in\Al_\infty$, then $\sub$ is tame.
	\end{cor}

The next corollary, which originally appeared in \cite{BKM:recognisable}, and which says that an aperiodic substitution is tame, follows from Theorem \ref{THM:wild-characterisation} and Lemma \ref{LEM:leftmost-rightmost-periodic}.  
\begin{cor}[Proposition 5.5 in \cite{BKM:recognisable}]\label{COR:aperiodic_implies_tame}
Let $\sub$ be a substitution on $\Al$. If $\sub$ is aperiodic, then $\sub$ is tame.
\end{cor}

Lemma \ref{LEM:leftmost-rightmost-periodic}, Theorem \ref{THM:wild-characterisation}, and Corollaries \ref{COR:leftmost-rightmost1} and \ref{COR:aperiodic_implies_tame} do not include minimality as a hypothesis, but in the presence of minimality there are further consequences.  
Any non-periodic minimal substitution is aperiodic, so Corollary \ref{COR:aperiodic_implies_tame} implies in particular that a non-periodic minimal substitution is tame.  

The following easy lemma is an immediate consequence of the definition of tameness (Definition \ref{DEF:wild}), and will be useful for rewriting tame minimal substitutions.  
\begin{lemma}\label{LEM:minimal-seed}
Let $\sub$ be a tame substitution on $\Al$.  If $X_\sub$ is non-empty, then it contains a bi-infinite sequence $w \in X_\sub$ with the property that there exists $M\in\N$ such that every word $u\subset w$ of length exceeding $M$ contains an expanding letter.  In particular, $w$ contains infinitely many expanding letters.  
\end{lemma}

\begin{remark}\label{REM:find-seed}
If $\sub$ is an aperiodic substitution on $\Al$, there is a recipe for finding legal expanding letters.  
Let $S$ denote the set of all pointed words in $\mathcal{L}_\sub$ of the form
\begin{equation}\label{EQ:seed}
u = b_{-k-1} a_{-k}\cdots a_{-1}.a_0\cdots a_lb_{l+1},
\end{equation}
where $b_{-k-1},b_{l+1}\in \Al_\infty$ and $a_i\in \Al_B$ for $-k\leq i \leq l$.  
Here the pointed words $a.bc$ and $ab.c$ are different.  
$S$ consists of words in the language of $\sub$, but they need not all be legal words, although by Lemma \ref{LEM:minimal-seed} $S$ contains at least one legal word.  

Define a function $f:S\to S$ as follows.  
For a pointed word $u$ of the form in \ref{EQ:seed}, the words $\sub(b_{-k-1})$ and $\sub(b_{m+1})$ are subwords of $w = \sub(u)$ occurring at the beginning and the end respectively.  
Each of these words contains at least one expanding letter.  
Let $b^-$ be the last such letter occurring in $\sub(b_{-k-1})$, and let $b^+$ be the first such letter occurring in $\sub(b_{m+1})$.  
Then $w$ contains a subsequence of the form $w_{m_1}\ldots w_{m_2}$, where $m_1 < 0 \leq m_2$, $w_{m_1} = b^-$, $w_{m_2} = b^+$, and $w_i \in \Al_B$ for all $m_1 < i < m_2$.  
Moreover, $w_{m_1} \ldots w_{m_2} \in S$.  
Therefore let us define $f(u) = w_{m_1} \ldots w_{m_2}$ (seen as a pointed word).  

Choose a word $u \in S$.  
Such a word can be found by considering the sets $\{ \sub^n(a) : a \in \Al\}$---by Lemma \ref{LEM:minimal-seed}, for sufficiently high $n$ this set contains a word with expanding letters in at least two positions, which can then be shifted and truncated to obtain $u\in S$.  

Consider the forward $f$-orbit of $u$.  
If this orbit were infinite, that would imply that $\sub$ satisfied the hypothesis of Lemma \ref{LEM:leftmost-rightmost-periodic}, using either $a = b_{-k-1}$ or $a = b_{l+1}$.  
This would imply that $X_\sub$ contained a periodic point, contradicting the hypothesis of aperiodicity.  
Therefore the forward $f$-orbit of $u$ is finite.  
This means that some word $v$ in this forward $f$-orbit is sent to itself under a power of $f$; $v$ can then be used as the seed to produce a bi-infinite word in $X_\sub$ that is fixed under $\sub$, and both of the expanding letters in $v$ are legal.  
\end{remark}

\subsection{A New Substitution}\label{SUBSEC:new-sub}

Any periodic minimal subshift is equal to the subshift of a primitive substitution of constant length (say, the substitution that sends each legal letter to the same sequence $u$ with the property that $\ldots u.uu\ldots$ is in the subshift), so in the periodic case the conclusion of Theorem \ref{THM:main} is immediately true.  
Therefore we may suppose henceforth that the minimal subshift $X_\sub$ is non-periodic, and hence tame, and so by Lemma \ref{LEM:minimal-seed} and Remark \ref{REM:find-seed}, that $X_\sub$ contains a bi-infinite sequence $w$ that is invariant under $\sub^N$ and that contains a legal letter $b \in \Al_\infty$.  
Moreover, $\sub^N(b)\subset \sub^N(w) \in X_\sub$ and $X_\sub$ is linearly recurrent, so for sufficiently large $N$, $\sub^N(b)$ must contain $b$.  
Similarly, linear recurrence implies that any legal word $u$ appears in $\sub^{k_uN}(b)$ for some $k_u\in\N$.  
By passing to a multiple of $N$ if necessary, we may suppose further that $\sub^N(b)$ contains at least two copies of the letter $b$.

Define $\Be := \{ bu : u \text{ does not contain } b \text{ and } bub\text{ is legal}\}$.  
In the terminology of \cite{D:main-result}, these are the \defemph{return words to $b$}.  

Enumerate the elements of $\Be\setminus \{ b\}$: $\Be\setminus \{ b\} = \{ v_1,\ldots , v_k\}$.  
If $b\in\Be$, then write $v_0 = b$.  

We can break $\sub^N(b)$ into block form:
\begin{align*}
\sub^N(b) & = uv_{01}\ldots v_{0r_0},
\end{align*}
where $u$ does not contain $b$ and, for $1\leq j \leq r_0$, $v_{0j}$ has the form $bv$ for some $v$ that does not contain $b$.  
Moreover, as $\sub^N(b)$ contains $b$ in at least two distinct places, we know that $r_0 > 1$.  
And, as $b$ is legal, so is $\sub^N(b)$, so if $j<r_0$ then the sequence $v_{0j}b$ is legal, and so $v_{0j}\in \Be$.  
The word $v_{0r_0}$ need not be in $\Be$.  

For each $i\geq 1$, we can write 
\begin{align*}
\sub^N(v_i) & = \sub^N(b)w_iv_{i1}\ldots v_{ir_i},
\end{align*}
where $r_i\geq 0$, $w_i$ does not contain $b$, and, for $1\leq j\leq r_i$, $v_{ij}$ has the form $bv$ for some $v$ that does not contain $b$.  
If $r_i > 0$, then for all $j < r_i$, the word $v_{ij}b$ appears in $\sub^N(v_i)$, and hence is legal, so $v_{ij} \in \Be$.  
The word $v_{ir_i}$ need not be in $\Be$, but $v_ib$ is legal, and hence $\sub^N(v_ib)$ is legal, and this word contains $v_{ir_i}ub$.  
Therefore, if $r_i>0$, then $v_{ir_i}u \in \Be$; let us denote this word by $v_{ir_i}'$.  

Further, although the word $v_{0r_0}$ from above need not be in $\Be$, for all $i$ with $r_i > 0$ it is true that $v_{0r_0}w_i \in \Be$, and for all $i$ with $r_i = 0$ it is true that $v_{0r_0}w_iu \in \Be$.  
Let us denote by $w_i'$ the word $v_{0r_0}w_i$ if $r_i>0$ or $v_{0r_0}w_iu$ if $r_i = 0$.  
Also $v_{0r_0}u \in \Be$; let us denote this word by $v_{0r_0}'$.  

Let $\Ga$ be a new alphabet, disjoint from $\Al$ and $\Be$, but with the same number of elements as $\Be$, and let $\alpha \colon \Be \to \Ga$ be a bijection of sets.  
The function $\alpha$ extends naturally to a map $\Be^+ \to \Ga^+$.  
For $v \in \Be$, let $\tilde{v}$ denote $\alpha(v)$.  
Define a substitution $\psi \colon \Ga \to \Ga^+$ by
\begin{align*}
\psi(\tilde{v}_0) & = \tilde{v}_{01}\ldots \tilde{v}_{0r_0-1}\tilde{v}_{0r_0}'
\intertext{ if $v_0 = b\in\Be$, and }
\psi(\tilde{v}_i) & = \left\{ \begin{array}{ll} \tilde{v}_{01}\ldots \tilde{v}_{0r_0-1}\tilde{w}_i'\tilde{v}_{i1}\ldots \tilde{v}_{ir_i-1}\tilde{v}_{ir_i}' & \text{ if } r_i>0 \\
\tilde{v}_{01}\ldots \tilde{v}_{0r_0-1}\tilde{w}_i' & \text{ if } r_i=0 
\end{array}\right.
\end{align*}
for all $i > 0$.  

\begin{lemma}\label{LEM:primitivity}
The substitution $\psi \colon \Ga \to \Ga^+$ defined above is primitive.  
\end{lemma}
\begin{proof}
For all $v\in\Be$ there exists $n_v\in\N$ such that $vb$ is a subword of $\sub^{n_vN}(b)$.  
But the hypothesis that $b$ is a subword of $\sub^N(b)$ means that, for all $k\leq l$, $\sub^{kN}(b)$ is a subword of $\sub^{lN}(b)$.  
Thus, picking $l = \max_{v\in\Be} n_v$ means that, for all $v\in\Be$, $vb$ is a word of $\sub^{lN}(b)$.  
Because all of the words $\{ vb : v\in\Be\}$ can be found in $\sub^{lN}(b)$, and because any two of these can have overlap in at most their first or last letters, it is possible to find all of the elements of $\Be$ as subwords of $\sub^{lN}(b)$, no two of which share any common indices.  

Moreover, for all $w \in \Be$, $\sub^N(w)$ starts with $uv_{01}$ and $b$ is a subword of $v_{01}$, so $\sub^{(l+1)N}(w)$ contains every $v \in \Be$ within the block $\sub^{lN}(v_{01})$ that begins at index $|\sub^{lN}(u)|$.  

Then for all $w \in \Be$, $\psi(\tilde{w})$ starts with $\tilde{v}_{01}$, so $\psi^{l+1}(\tilde{w})$ contains $\tilde{v}$ for all $v \in \Be$.  
Therefore $\psi$ is primitive.  
\end{proof}

\subsection{Topological Conjugacy}\label{SUBSEC:topological-conjugacy}

The new substitution $\psi$ is related to $\sub$ (specifically, they give rise to homeomorphic tiling spaces---see Section \ref{SUBSEC:subshifts}), but it does not necessarily give rise to a topologically conjugate subshift.  
For this the following result, proved in \cite[Proposition 3.1]{D:main-result} and paraphrased here, will be useful.  

\begin{prop}\label{PROP:letter-expansion}
Let $\psi \colon \Ga \to \Ga^+$ be a primitive substitution and let $g$ be a map from $\Ga$ to $\Al^+$.  
Let $X_g \subset \Al^\Z$ denote the subshift generated by $g(X_\psi)$---that is, $X_g := \{ \sigma_\Al^n(g(x)) : x \in X_\psi, n \in \Z\}$.  
Then there exists an alphabet $\Ze$, a primitive substitution $\theta \colon \Ze \to \Ze^+$, and a map $h \colon \Ze \to \Al$ such that $h(X_\theta) = X_g$.
\end{prop}

We can apply this result to the current setting by using the substitution $\psi \colon \Ga \to \Ga^+$ defined above, which was shown to be primitive in Lemma \ref{LEM:primitivity}, and the map $g \colon \Ga \to \Al^+$ defined by $g(\tilde{v}_i) = v_i$, where $v_i \in \Al^+$ is viewed as a word possibly consisting of more than one letter.  
Then the subshift $X_g$ from the statement of Proposition \ref{PROP:letter-expansion} is exactly the original substitution subshift $X_\sub$.  
Therefore Proposition \ref{PROP:letter-expansion} guarantees the existence of a factor map---in fact, a one-block code \cite{LM:book}---from a primitive substitution subshift $X_\theta$ to the given minimal substitution subshift $X_\sub$.  

If we look at how $\Ze$ and $\theta$ are defined in the proof of Proposition \ref{PROP:letter-expansion} in \cite{D:main-result}, then it becomes clear that the factor map $h$ is in fact a topological conjugacy---i.e., it has an inverse that is also a factor map.  
Indeed, $\Ze$ is the set of all pairs $(\tilde{v},k)$, where $v \in\Be$ and $1 \leq k \leq |v|$.  
Every sequence $w \in X_\sub$ can be represented uniquely as a concatenation of return words $v \in \Be$ (with the origin possibly contained in the interior of such a word).  
Then there is a map $p: X_\sub\to \Ze^\Z$ defined in the following way on a sequence $w\in X_\sub$: If $w_j$ falls at position $k$ in the return word $v_i$, then $p(w)_j = (\tilde{v}_i,k)$.  
This is a sliding block code with block size equal to $\max_{v \in \Be}|v|$, and the one-block code $h$ is its inverse.  
The usefulness of Proposition \ref{PROP:letter-expansion} is in showing that $p(X_\sub)$ is in fact a primitive substitution subshift, which completes the proof of Theorem \ref{THM:main}.

\section{Examples and Applications}\label{SEC:examples}

The primitive substitution subshift $X_\theta$ is topologically conjugate to the original minimal subshift $X_\sub$, which is a very strong condition, but this comes at a price: if we follow the recipe from \cite[Proposition 3.1]{D:main-result} strictly, then the new alphabet $\Ze$ may be quite large---see the discussion after Proposition \ref{PROP:cohomology-rank}, below, for an example in which $|\Al| = 2$, $|\Ga| = 3$ and $|\Ze| = 9$.  
For some computational purposes, particularly purposes involving tiling spaces, the substitution $\psi\colon \Ga \to \Ga^+$ can be just as good as $\theta$, and typically uses a smaller alphabet.  

Consider the substitutions $\sub \colon \Al \to \Al^+$ and $\psi \colon \Ga \to \Ga^+$ from Theorem \ref{THM:main}, and the map $\alpha \colon \Ga \to \Be^+ \subset \Al^+$.  
Then the tiling spaces $\TS_\sub$ and $\TS_\psi$ are homeomorphic via the map 
\begin{align*}
f \colon \TS_\psi & \to \TS_\sub \\
(w,t) & \mapsto (\sigma^{\lfloor \tilde{t} \rfloor}(\alpha(w)), \tilde{t}-\lfloor \tilde{t} \rfloor),
\end{align*}
where $\tilde{t}=|\alpha(w_0)|\cdot t$.  

This means that, for practical purposes, we can use $\TS_\psi$ to compute the topological invariants of $\TS_\sub$.  
This is the approach in the following examples and applications, which illustrate the construction outlined in Section \ref{SEC:minimal}.  
The first example illustrates some of the greater `freedom' in behaviour exhibited by minimal non-primitive substitutions on small alphabets.
	
	Recall from \cite{GM:multi-one-d} and \cite{R:mixed-subs} that, if $\TS_\sub$ is the tiling space associated to a primitive substitution $\sub$ on an alphabet $\Al$ with $k$ letters, the rank of the first \Cech cohomology $\check{H}^1$ of $\TS_\sub$ is bounded above by $k^2-k+1$ and this bound is tight.  Recall from \cite{BDH:asymp-orbits} that $X_\sub$ has at most $k^2$ asymptotic orbits (equivalently, $\TS_\sub$ has at most $k^2$ asymptotic arc components) and this bound is tight.  These results both fail spectacularly if we allow for non-primitive minimal substitutions---this result suggests that the alphabet size is not as much of a limiting factor with respect to the topological and dynamical properties of a substitution.
	\begin{prop}\label{PROP:cohomology-rank}
		Let $\Al = \{a, b\}$ be an alphabet on only two letters. For all $n \geq 2$ there exists a minimal substitution $\sub_n \colon \Al \to \Al^+$ such that $\check{H}^1(\TS_{\sub_n})$ has rank $n$ and $X_{\sub_n}$ has at least $n$ asymptotic orbits.
	\end{prop}	
	\begin{proof}
	We construct $\sub_n$ explicitly and use the methods from Section \ref{SEC:minimal} to prove the claim.	
	
		We define our family of substitutions $\sub_n$ by
		$$
		\sub_n \colon
			\left\{
				\begin{array}{rcl}
					a & \mapsto & ab \, ab^2 \, \ldots ab^{n-1} \, ab^n \, a \\
					b & \mapsto & b
				\end{array}
			\right.
		$$
		We leave confirmation of minimality of the substitution $\sub_n$ to the reader. The decomposition $\Al = \Al_{\infty} \sqcup \Al_B = \{a\} \sqcup \{b\}$ is quickly found and, as $\sub_n$ satisfies the hypotheses of Corollary \ref{COR:leftmost-rightmost1} we know that $a$ can be used as the seed letter for our return words. The set $S = \{ ab^i.b^ja \mid i+j\leq n,\: 0\leq i, j \}$ defined in Remark \ref{REM:find-seed} has a fixed point under $f^1 (=\mbox{Id}_S)$, and $\sub_n(a)$ contains at least two distinct copies of $a$ so we can choose $N=1$.
		
		The return words to $a$ are $\Be_n = \{ ab^i \mid 1\leq i\leq n \}$. Let $v_i = ab^i$. The word $\sub(a)$ can be written as $u v_{01} \ldots v_{0r_0}$ with $u = \epsilon$ the empty word, $r_0 = n+1$, $v_{0i} = ab^i = v_i$ for $1 \leq i \leq n$ and $v_{0r_0} = a$. For each $1 \leq i \leq n$ we can write $\sub(v_i) = \sub(a) w_i$ with $w_i = b^i$ and we note that $r_i = 0$ for each $i\geq 1$. So, $w'_i = v_{0r_0} w_i u = a b^i = v_i$.
		
		We form $\Ga_n = \{ \tilde{v} \mid v \in \Be_n \} = \{ \tilde{v}_i \mid 1\leq i \leq n\}$ and define, for each $1 \leq i \leq n$ the substitution $\psi_n \colon \Ga_n \to \Ga_n^+$ by
		$$
			\begin{array}{rcl}
				\psi_n(\tilde{v}_i) & = & \tilde{v}_{01} \tilde{v}_{02} \ldots \tilde{v}_{0n} \tilde{w}'_{i} \\
				& = & \tilde{v}_1 \tilde{v}_2  \ldots \tilde{v}_n \tilde{v}_i.
			\end{array}
		$$
		This can more succinctly be written on the alphabet $\{1,\ldots, n\}$ as
		$$\psi_n\colon i \mapsto 1 2  \ldots n i.$$

		By Lemma \ref{LEM:primitivity} and the discussion in Section \ref{SUBSEC:subshifts}, $\psi_n$ is a primitive substitution whose tiling space $\TS_{\psi_n}$ is homeomorphic to $\TS_{\sub_n}$.  By the non-periodicity of $\sub_n$ and a result of Moss\'{e} \cite{M:aperiodic}, $\psi_n$ is recognisable. We notice that $\psi_n$ is also a left-proper substitution and so by standard results about left-proper substitutions (see \cite{S:book}), the first \Cech cohomology of $\TS_{\psi_n}$ (and hence of $\TS_{\sub_n}$) is given by the direct limit of the transpose of the incidence matrix of $\psi_n$ acting on the group $\Z^n$.

		The incidence matrix of $\psi_n$ is the symmetric matrix $M_n = \mathbf{1}_n + I_n$ where $\mathbf{1}_n$ is the $n\times n$ matrix of all $1$s, and $I_n$ is the $n\times n$ identity matrix. It is easy to check that $M_n$ has full rank and so $$\mbox{rk} \check{H}^1(\TS_{\sub_n}) = \mbox{rk} \check{H}^1(\TS_{\psi_n}) = \mbox{rk} \varinjlim(M_n) = n.$$
		
		To prove the claim about asymptotic orbits, we note that there exists a right infinite sequence $v$ such that for every $i\in \{1,2,\ldots, n\}$ there exists a left infinite sequence $u_i$ (found by repeated substitution on the sequence $i.1$) such that $u_ii.1v$ is a point in $X_{\psi_n}$. By construction then, $u_ii.1v = u_jj.1v$ if and only if $i=j$, and the bi-infinite sequences $u_ii.1v$ and $u_jj.1v$ agree on all indices right of the origin for all pairs $i,j$. It follows that each pair $i,j$ leads to a right asymptotic pair of orbits in $X_{\psi_n}$ and so there exist at least $n$ asymptotic orbits.
		
		Equivalently then, $\TS_{\psi_n}$ has at least $n$ asymptotic arc components. These are preserved under homeomorphism and so $\TS_{\sub_n}$ also has at least $n$ asymptotic arc components. Equivalently, $X_{\sub_n}$ has at least $n$ asymptotic orbits.
	\end{proof}

	We emphasise that our algorithm for producing a primitive substitution with a homeomorphic tiling space is much more computationally simple than the corresponding algorithm for the substitution with a topologically conjugate subshift. As an example, compare the presentation and relatively short computation of the substitution $\psi_3$ in the above proof with the following. The reader is invited to verify, following the proof of \cite[Proposition 3.1]{D:main-result}, that the substitution $\theta$ defined by
$$
\begin{array}{lcllcllcl}
\theta(A) & = & AB        & \theta(L) & = & AB      & \theta(W) & = & AB     \\
\theta(B) & = & LMNWXYZAB & \theta(M) & = & LMN     & \theta(X) & = & LMN    \\
          &   &           & \theta(N) & = & WXYZLMN & \theta(Y) & = & WXYZ   \\
          &   &           &           &   &         & \theta(Z) & = & WXYZ   
\end{array}
$$
produces a subshift that is topologically conjugate to $\sub_3$, where the conjugacy $$h \colon \{A,B,L,M,N,W,X,Y,Z\} \to \{a,b\}$$ is given by $h(A) = h(L) = h(W) = a$ and $h(B) = h(M) = h(N) = h(X) = h(Y) = h(Z) = b$.  
(Of course, it is clear that a smaller alphabet can be used; this is what is obtained when the recipe is followed without modification.)

	The reader is encouraged to try the example $\sub \colon a\mapsto acb,\: b\mapsto adb,\: c\mapsto dd,\: d\mapsto d$ where the function $f$ is not a bijection; and to also try the example of the non-primitive Chacon substitution $\sub \colon a\mapsto aaba,\: b\mapsto b$ where $\Be$ contains the single letter return word $a$ and so $v_0$ needs to be treated.
	
	\section{The Non-Minimal Case}\label{SEC:non-minimal}

	In this section, we now turn our attention to the case of those substitutions which give rise to non-minimal subshifts. We call such substitutions \emph{non-minimal substitutions}.

In the primitive case, a standard approach is to replace the alphabet $\Al$ with a new alphabet consisting of \defemph{collared letters}, which are copies of the letters from $\Al$ containing extra information about the letters that appear around them in elements of $X_\sub$.  
It is important to be careful about how we extend this idea to the non-minimal case, as a non-minimal substitution is non-primitive, so in particular it may have letters in its alphabet that do not appear in any element of $X_\sub$.  
The $n$-collared alphabet, which we will define below, is a subset of $\Al^{2n+1}$, where a $(2n+1)$-letter word should be thought of as a single letter---the one at position $n+1$---along with extra information about the $n$ letters that lie to either side of it.  

View the words in $\Al^{2n+1}$ as being indexed so that their middle letter is at position $0$.  
Let $a \in \Al$ and let $u = a_{-n}\ldots a_{-1} a a_1 \ldots a_n \in \Al^{2n+1}$ be a word whose central letter is $a$; then we define an \emph{$n$-collared letter} to be this formal pair $(a,u)$ and denote it by $a_u$.

	If $u$ is the word $ c_1 c_2 \ldots c_l \in \Al^*$, then for those $i$ and $n$ where it is well-defined, let $c_i(n)$ be the subword $c_{i-n} \ldots c_i \ldots c_{i+n}$.
	Suppose $\sub(a) = b_1 \ldots b_k$ and let $a_u$ be an $n$-collared letter. Note that $b_1 \ldots b_k$ is a subword of $\sub(u)$, so we can define $b_i(n)$ for each $1\leq i \leq k$.

There is an induced substitution $\sub_n'$ defined on $\Al^{2n+1}$ by 
	$$\sub_n'(a) =
	({b_1})_{b_1(n)}\ldots ({b_k})_{b_k(n)}.$$

Let us define an $n$-collared substitution, $\sub_n$, by restricting $\sub_n'$ to a sub-alphabet $\Al_n\subset \Al^{2n+1}$. The definition of such a subalphabet requires some work.

In defining the collared version of a substitution, it is tempting to take as a new alphabet the set of all collared letters $a_u$ such that $u$ is a legal word for $\sub$. The problem is that this alphabet essentially ignores the illegal letters in the original alphabet. Illegal letters are still crucial in generating pieces of the subshift which nevertheless do not contain the illegal letter. We are then in the position of having to include collared versions of the illegal letters into our new alphabet in a well-defined and consistent manner. There are various valid methods of doing this and the following is only one approach of many.

We may suppose $\Al$ contains at least one letter, say $b$, otherwise $X_\sub$ is empty.  
Then we define $\Al_n$ and $\sub_n$ in terms of this $b$ (so they are not, in general, unique, although we will show below in Proposition \ref{PROP:collared-substitution} that the resulting subshift is always topologically conjugate to $X_\sub$).

\begin{defn}\label{DEF:collar}
Let $n$ be a non-negative integer and let $\sub$ be a substitution on $\Al$.  
Define the following two sets.  
\begin{align*}
\Al_n' & := \{ a_u\in \Al^{2n+1} : u \text{ is legal } \} \\
\Al_n'' & := \{ a_u \in \Al^{2n+1} : u = a_{-n}\cdots a_n,  a_i = b \text{ for all } i\neq 0, a = a_0  \text{ is illegal }\}.   
\end{align*}
Then define the \defemph{$n$-collared alphabet} 
\[
\Al_n := \{ a_u \in \Al^{2n+1} : \text{ there exists } k\geq 0 \text{ such that } a_u \subset \sub_n'^k(c_v) \text{ for some } c_v\in \Al_n' \cup \Al_n''\}.  
\]
Define the \defemph{$n$-collared substitution} $\sub_n$ to be the restriction of $\sub_n'$ to $\Al_n$.  

These definitions depend upon the choice of letter $b$, but let us suppose that a letter $b$ has been chosen, and that, for all $n$, $\Al_n$ and $\sub_n$ have been defined using this letter.  
\end{defn}

It is clear from the way that $\Al_n$ is defined that $\sub_n(a_u) \in \Al_n^*$ for all $a_u\in\Al_n$, so $\sub_n$ is indeed a substitution on $\Al_n$.  
Next let us show that $X_{\sub_n}$ is topologically conjugate to $X_\sub$.  

	Every bi-infinite sequence in $X_\sub$ can be rewritten using $n$-collared letters using the local rule that if $u$ is the $(2n+1)$-letter word which contains the symbol $a$ at its center, then this instance of $a$ is replaced by $a_u$. This map embeds the subshift $X_\sub$ into the full shift $(\Al_n')^\Z\subseteq \Al_n^\Z$. Call this embedding $\iota_n \colon X_\sub \to \Al_n^\Z$. This map $\iota_n$ is clearly a topological conjugacy onto its image (the inverse is given by forgetting the collaring $a_u \mapsto a$).

\begin{prop}\label{PROP:collared-substitution}
$X_{\sub_n} = \iota_n(X_\sub)$.  
\end{prop}
\begin{proof}
$\iota_n$ is injective, so to prove the claim it will suffice to show that it maps $X_\sub$ surjectively onto $X_{\sub_n}$.  
Pick $w\in X_\sub$ and consider its image $\iota_n(w)$.  
To show that this image is in $X_{\sub_n}$, it is enough to show that, for an arbitrary $n$-collared word $u\subset \iota_n(w)$ it is true that $u$ is in the language of $\sub_n$.  

Say that $u$ is the subword of $\iota_n(w)$ that begins at index $i$ and ends at index $j>i$.  
Consider the word $v\subset w$ that begins at index $i-n$ and ends at index $j+n$.  
As $v\subset w$, it is in the language of $\sub$, and so there is some letter $a\in \Al$ and $k\geq 0$ such that $v\subset \sub^k(a)$.  
$\Al_n$ contains collared versions of all legal and illegal letters, so there is a collared letter $a_s \in \Al_n$.  
By construction, $\sub_n^k(a_s)$ contains the collared word $u$.  
Since $u\subset \iota_n(w)$ was arbitrary, $\iota_n(w)\in X_{\sub_n}$.  
Thus $\iota_n(X_\sub)\subseteq X_{\sub_n}$.  

For $w'\in X_{\sub_n}$, let $w\in\Al^\Z$ be the word obtained from $w'$ by forgetting collars.  
To prove that $\iota_n$ is surjective, it will suffice to show that, for arbitrary $w'\in X_{\sub_n}$, it is true that $w\in X_\sub$.  

Pick an arbitrary $u\subset w$.  
Say that $u$ is the subword of $w$ that begins at index $i$ and ends at index $j>i$.  
Consider the word $u'\subset w'$ that begins at index $i-n$ and ends at index $j+n$.  
As $u'\subset w'$, it is in the language of $\sub_n$, and so there is some collared letter $a_s\in \Al_n$ and $k\geq 0$ such that $u'\subset \sub_n^k(a_s)$.  
But then $\sub^k(a)$ contains the word $u$, and hence $u$ is in the language of $\sub$.  
Since $u\subset w$ was arbitrary, this means $w\in X_\sub$, and so $\iota_n$ is surjective.  
\end{proof}

	Under the assumption that $\Al_n$ and $\Al_m$ are defined using the same letter $b$, there is a forgetful map $f_{n,m} \colon \Al_n \to \Al_m$ where, if $$u = a_{-n} \ldots a_{-m} \ldots a_{-1} a a_1 \ldots a_m \ldots a_n$$ and $$v = a_{-m} \ldots a_{-1} a a_1 \ldots a_m $$ then we define $f_{n,m}(a_u) = a_v$.
	We can extend this forgetful map to $\Al_n^*$ and $\Al_n^\Z$.
	
	This $n$-collared substitution commutes with the forgetful maps.
	That is, $$\sub_m \circ f_{n,m} = f_{n,m} \circ \sub_n.$$
	If $l<m<n$, then $f_{m,l}\circ f_{n,m} = f_{n,l}$.  
	Note that $\Al_0 = \Al$ and $\sub_0 = \sub$, and by Proposition \ref{PROP:collared-substitution}, $f_{n,0}$ is a topological conjugacy, from which it follows that $f_{n,m}\colon X_{\sub_n} \to X_{\sub_m}$ is a topological conjugacy for all $n \geq m \geq 0$.

Recall the following definition.
A substitution \defemph{forces the border at level $k$} if every appearance of $\sub^k(a)$ as a $k$-supertile extends uniquely for all $a\in \Al$.  That is, if $w \in X_\sub$ is an admitted bi-infinite sequence and it is decomposed into words of the form $\sub^k(a)$ for letters $a\in\Al$, then for every $a$, there exist unique letters $a_l, a_r$ such that every appearance of $\sub^k(a)$ in the decomposition of $w$ sits as an interior subword of a word of the form $a_l\sub^k(a)a_r$.
	
	\begin{lemma}\label{LEM:border-forcing}
		Let $\sub$ be a tame substitution. Let $N$ be one greater than the maximum length of any bounded word for $\sub$, $N = \max_{u\in B} |u|+1$. The substitution $\sub_N \colon \Al_N\to \Al_N^+$ forces the border at some level $k$.
	\end{lemma}
	\begin{proof}
		Let $k$ be such that, for every expanding letter $a \in \Al$, we have $|\sub^k(a)| > N$. Let $a_u \in \Al_N$ be an $N$-collared letter such that $\sub^k_N(a_u)$ appears as a subword of $w \in X_{\sub_N}$. By our choice of $N$, there exists a letter $l = w_{i-j}$ to the left of $a$ in $u$ and a letter $r = w_{i+j'}$ to the right of $a$ in $u$ that are both expanding letters. Further, the indices $j$ and $j'$ can be chosen so that $j,j' < N$.
		
		So, $u = w_{i-N}\ldots l \ldots a \ldots r \ldots w_{i+N}$ where $w_i = a$ is the central letter of the word and then $$\sub^k(u) = \sub^k(w_{i-N})\ldots \sub^k(l) \ldots \sub^k(a) \ldots \sub^k(r) \ldots \sub^k(w_{i+N})$$
		
		As $|\sub^k(l)| > N$ and $|\sub^k(r)| > N$, and we know all tiles within $N$ places of $a$ in $u$, we can determine all $N$-collared tiles out until at least the rightmost $N$-collared letter of $\sub^k(l)$ to the left of any appearance of $\sub^k(a_u)$ in the decomposition of $w$ as $k$-supertiles, and at least the leftmost $N$-collared letter of $\sub^k(r)$ to the right of any appearance of $\sub^k(a_u)$ in the decomposition of $w$ as $k$-supertiles. These tiles lie outside of $\sub^k_N(a_u)$ and so $\sub^k_N(a_u)$ uniquely extends as a $k$-supertile in $X_{\sub_N}$. It follows that $\sub_N$ forces the border at level $k$.
	\end{proof}

		Let $\sub$ be a substitution on the alphabet $\Al$ and let $\TS$ be the associated tiling space. Use the convention that a point $T\in \TS$ is written coordinate-wise as $(w,t)$, $w \in X_\sub$ and $t \in [0,1)$. Recall \cite{AP} that we define the \emph{Anderson-Putnam complex} $\Gamma$ of $\sub$ to be $\TS/{\sim}$ where $\sim$ is the equivalence relation given by taking the transitive closure of the relation $(w,t)\sim (w',t')$ if $t = t' \in (0,1)$ and $w_0 = w'_0$ or $t = t' = 0$ and $w_{-1} = w'_{-1}$ or $w_0 = w'_0$.
		\begin{defn}
		We define the \emph{$n$-collared Anderson Putnam complex} $\Gamma_n$ to be the Anderson-Putnam complex associated to the $n$-collared substitution $\sub_n$.
		
		Let $p_n \colon \TS \to \Gamma_n$ be the natural quotient map. We define a map $f_n \colon \Gamma_n \to \Gamma_n$ to be the unique map which makes the following square  commute
		$$
			\begin{tikzpicture}[node distance=2.5cm, auto]
  				\node (00) {$\Gamma_n$};
  				\node (10) [node distance=2cm, above of=00] {$\TS$};
  				\node (01) [right of=00] {$\Gamma_n$};
  				\node (11) [right of=10] {$\TS$};

  				\draw[->] (10) to node [swap] {$p_n$} (00);
  				\draw[->] (11) to node [swap] {$p_n$} (01);
  				\draw[->] (01) to node [swap] {$f_n$} (00);
  				\draw[->] (11) to node [swap] {$\sub$} (10);
  				\draw (11) to node  {$\cong$} (10);
			\end{tikzpicture}
		$$
	\end{defn}
	For a tame substitution $\sub$, let $N_\sub = \max_{u\in B} |u|+1$ be one greater than the length of the longest legal bounded word for $\sub$. The following theorem allows us to replace the hypothesis of primitivity with tameness for the classic Anderson-Putnam Theorem \cite{AP} if we allow ourselves to collar letters out to a sufficient radius.
	\begin{thm}\label{THM:homeo}
	Let $\sub$ be a tame recognisable substitution. The natural map $h \colon \TS \to \varprojlim(\Gamma_{N_\sub}, f_{N_\sub})$ given by $$h(x) = (p_{N_\sub}(x), p_{N_\sub}(\sub^{-1}(x)), p_{N_\sub}(\sub^{-2}(x)), \ldots )$$ is a homeomorphism.
	\end{thm}
	\begin{proof}
	Recognisability of $\sub$ means that $\sub \colon \TS \to \TS$ has a well-defined inverse and so $h$ is well-defined. By the choice of $N_\sub$ and Lemma \ref{LEM:border-forcing}, the $N_\sub$-collared substitution $\sub_{N_\sub}$ forces the border at level $k$. Hence, a point in the inverse limit describes a unique tiling of the line, and so $h$ is both injective and surjective.	As $h$ is a continuous bijection from a compact space to a Hausdorff space, $h$ is a homeomorphism.
	\end{proof}
	If we are only concerned with aperiodic substitutions (which are tame by Corollary \ref{COR:aperiodic_implies_tame}), then we may further reduce the list of hypotheses for this theorem by making use of a result of Bezuglyi, Kwiatowski and Medynets \cite{BKM:recognisable}.
	\begin{thm}[Bezuglyi--Kwiatowski--Medynets]
	If $\sub$ is aperiodic, then $\sub$ is recognisable.
	\end{thm}
	\begin{cor}
	Let $\sub$ be an aperiodic substitution. The map $h \colon \TS \to \varprojlim(\Gamma_{N_\sub}, f_{N_\sub})$ is a homeomorphism.
	\end{cor}

	\begin{remark}
	There exist recognisable substitutions which are not aperiodic or even non-aperiodic. Take as an example the substitution $\sub \colon a \mapsto ab,\: b \mapsto b$ from Example \ref{EX:periodic}, whose induced substitution on the tiling space is just the identity map on a circle, and so is injective, hence $\sub$ is recognisable even though $\sub$ is a periodic substitution. This is perhaps surprising to a reader who is used to primitive substitutions, where recognisability, non-periodicity and aperiodicity are all equivalent.
	\end{remark}	
	\section{Closed Invariant Subspaces of Non-minimal Tiling Spaces}\label{SEC:CISs}
	\subsection{Invariant Subspaces}
	
	Let $\TS$ be a compact metric space and let $G$ act continuously on the right of $\TS$ via $\rho\colon \TS \times G\to \TS$ and let $\rho_{\tau}\colon \TS \to \TS$ be given by $x \mapsto \rho(x,\tau)$. We will normally only consider $G = \R$ or $G = \Z$, but the following machinery is applicable in the general case (in particular, tiling spaces in higher dimensions which have actions of higher dimensional Euclidean groups).
	
	If $\Lambda$ is a closed subspace of $\TS$ such that $\rho_\tau(\Lambda) = \Lambda$ for all $\tau \in G$, we call $\Lambda$ a \emph{closed invariant subspace} with respect to the action, or \emph{CIS} for short. The set of CISs $\mathcal{C}$ forms a lattice under inclusion of subspaces. The least elements of $\mathcal{C}$ that are not empty are the minimal sets of the action on $\TS$. The unique maximal element of $\mathcal{C}$ is the whole space $\TS$. Let us observe, without making further comment, the interesting fact that $\mathcal{C}^C = \{\TS \setminus \Lambda \mid \Lambda \in \mathcal{C}\}$ is a topology on the set $\TS$ (in general, coarser than the original topology induced by the metric on $\TS$). This topology is indiscrete if and only if $(\TS,\rho)$ is minimal. Any continuous map between dynamical systems which maps orbits to orbits will also induce a continuous map between the spaces endowed with the topology $\mathcal{C}^C$, and so the homeomorphism type of the topological space $(\Omega,\mathcal{C}^C)$ is an orbit equivalence invariant of the dynamical system $(\Omega,\rho)$.  
	
		\begin{lemma}\label{LEM:orbits}
			Let $f \colon \TS \to \TS'$ be a continuous map which maps $G$-orbits to $G$-orbits. If $\Lambda$ is a CIS of $\TS$ with respect to the action of $G$ on $\TS$, then $f(\Lambda)$ is a CIS of $\TS'$ with respect to the action of $G$ on $\TS'$.
			
		\end{lemma}
		\begin{proof}
			Let $\Lambda$ be a CIS of $\TS$. As $\TS$ is compact and $\Lambda$ is a closed subspace, $\Lambda$ is compact, so the image of $\Lambda$ under a continuous map is compact. As $\TS'$ is Hausdorff, a compact subspace of $\TS'$ must be closed, and so $f(\Lambda)$ is a closed subspace of $\TS'$.
			
			Let $\mathcal{O}_x = \{ \rho_\tau(x) \mid \tau \in G\}$ be the orbit of a point $x \in \TS$ under the $G$-action. From the definition of a CIS, if $x \in \Lambda$ then $\rho_\tau(x) \in \Lambda$ for all $\tau$ and so $\Lambda$ contains $\mathcal{O}_x$ for all points $x\in \Lambda$. If $y\in f(\Lambda)$, then there exists an $x \in \Lambda$ such that $f(x) = y$. The image of an orbit under $f$ is also an orbit, and so as $f(\mathcal{O}_x) \subset f(\Lambda)$, and as $y$ is a point on that orbit, we find that $f(\mathcal{O}_x) = \mathcal{O}_y$ and $\mathcal{O}_y \subset f(\Lambda)$. Hence $f(\Lambda)$ is invariant under the action of $G$, and so forms a CIS.
		\end{proof}
		Let $\TS$ be a compact metric space on which the group $G$ acts on the right and let $\mathcal{C}$ be the set of CISs for $\TS$.
		\begin{defn}
			The \emph{inclusion diagram} $D_{\TS}$ for $\TS$ is a diagram whose objects are the elements of $\mathcal{C}$ and whose arrows $i_{jk} \colon \Lambda_j \to \Lambda_k$ are given by inclusion for every pair $j, k$ such that $\Lambda_j \subset \Lambda_k$.
			
			The \emph{inclusion cohomology diagram} of $\TS$, denoted by $\check{H}^*(D_\TS)$, is given by the diagram of groups induced by applying the \Cech cohomology functor to $D_{\TS}$.
		\end{defn}
		\begin{defn}
			The \emph{quotient diagram} $D^{\TS}$ for $\TS$ is a diagram with objects $\TS/\Lambda$ for every $\Lambda \in \mathcal{C}$ and an arrow $q_{jk} \colon \TS/\Lambda_j \to \TS/\Lambda_k$ given by the quotient map for every pair $j,k$ such that $\Lambda_j \subset \Lambda_k$.
			
			The \emph{quotient cohomology diagram} of $\TS$, denoted by $\check{H}^*(D^\TS)$, is given by the diagram of groups induced by applying the \Cech cohomology functor to $D^{\TS}$.
		\end{defn}
		These diagrams clearly have the same shape given by the correspondence between $\Lambda$ and $\TS/\Lambda$ for every CIS $\Lambda$.
		\begin{example}
		Consider the shift-orbit $\mathcal{O}(w)$ of the sequence $w = \ldots aaaa.bcbcbc\ldots$. The closure of $\mathcal{O}(w)$ in $\{a,b,c\}^\Z$ also contains three additional sequences in the limit and so the subshift $X_w$ generated by $w$ is $$X_w = \mathcal{O}(w) \sqcup \Lambda_1 \sqcup \Lambda_2$$ where $\Lambda_1 = \{\ldots aaa.aaa\ldots\} $ and $\Lambda_2 = \{\ldots bcbc.bcbc\ldots,\ldots cbcb.cbcb \ldots\}$.
		The dynamical system has two minimal sets $\Lambda_1$ and $\Lambda_2$ and one other proper non-empty closed invariant subspace given by their disjoint union $\Lambda_3 = \Lambda_1 \sqcup \Lambda_2$. The inclusion and quotient diagrams $D_{X_w}$ and $D^{X_w}$ for this dynamical system are given below, with the composition of arrows being implicit.
				$$
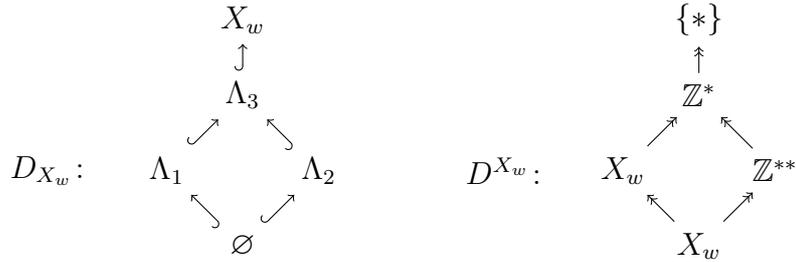
\begin{tikzpicture}[node distance=1cm, auto]
  		\node (1) {};
  		\node (Y) [below of=1] {};
  		\node (X) [right of=1] {$X_w$};
  		\node (2) [below of=X] {$\Lambda_3$};
  		\node (3) [right of=2] {};
  		\node (4) [below of=Y] {$\Lambda_1$};
  		\node (5) [right of=4] {};
  		\node (6) [right of=5] {$\Lambda_2$};
  		\node (7) [below of=4] {};
  		\node (8) [right of=7] {$\emptyset$};
  		\node (9) [right of=8]{};
  		\draw[right hook->] (4) to node [swap] {} (2);
  		\draw[right hook->] (6) to node [swap] {} (2);
  		\draw[right hook->] (8) to node [swap] {} (4);
  		\draw[right hook->] (8) to node [swap] {} (6);
  		\draw[right hook->] (2) to node [swap] {} (X);
  		
  		\node (A) [left of=4, node distance=1.5cm] {$D_{X_w}\colon$};
  		
  		\node (1b) [right of=1, node distance=6cm]{};
  		\node (Yb) [below of=1b] {};
  		\node (Xb) [right of=1b] {$\{\ast\}$};
  		\node (2b) [below of=Xb] {$\Z^*$};
  		\node (3b) [right of=2b] {};
  		\node (4b) [below of=Yb] {$X_w$};
  		\node (5b) [right of=4b] {};
  		\node (6b) [right of=5b] {$\Z^{**}$};
  		\node (7b) [below of=4b] {};
  		\node (8b) [right of=7b] {$X_w$};
  		\node (9b) [right of=8b]{};
  		\draw[->>] (4b) to node [swap] {} (2b);
  		\draw[->>] (6b) to node [swap] {} (2b);
  		\draw[->>] (8b) to node [swap] {} (4b);
  		\draw[->>] (8b) to node [swap] {} (6b); 
  		\draw[->>] (2b) to node [swap] {} (Xb);
  		
  		\node (Ab) [left of=4b, node distance=1.5cm] {$D^{X_w}\colon$};
		\end{tikzpicture}$$
		where $\Z^*$ is the one-point compactification of the integers whose point at $\infty$ is the equivalence class $[\Lambda_3]$ shrunk to a point and where $\Z^{**}$ is the two-point compactification of the integers $\Z\cup\{-\infty,\infty\}$ whose point at $-\infty$ is the point $\{\ldots aaa.aaa \ldots\}$ and whose point at $\infty$ is the equivalence class $[\Lambda_2]$ shrunk to a point. Note that one of the copies of $X_w$ in the right hand diagram is really $X_w$ but with the single point $\Lambda_1$ replaced (or rather renamed) with the equivalence class $[\Lambda_1]$.
		\end{example}
		\begin{example}
		The inclusion and quotient diagrams for a compact metric space need not be finite. Take the toral product system $\TS = S^1\times S^1$ with rotation action $\rho \colon \TS \times\R \to \TS$ given by $\rho((\theta_1,\theta_2),t) = (\theta_1, \theta_2+t)\mod 1$. For every $\theta\in S^1$, the meridional circle $\{\theta\}\times S^1$ is closed and invariant (and in fact minimal). Moreover, every CIS of $(\Omega,\rho)$ is of the form $C\times S^1$ for some closed set $C$ and so the inclusion and quotient diagrams are isomorphic (as lattices) to the lattice of closed subsets of the circle $S^1$.
		
		More generally, if every $\rho$-orbit of the compact dynamical system $(\TS,\rho)$ is closed, then the lattice of CISs $\mathcal{C}$ is naturally isomorphic to the lattice of closed sets for the orbit space $\TS/\rho$ under the map taking every CIS to the equivalence class represented by that closed union of orbits.
		\end{example}
		
		\begin{remark}
			Note that all of the arrows appearing in $D_\TS$ and $D^\TS$ commute with the $G$-action induced on the objects $\Lambda$ and $\TS/\Lambda$ for all $\Lambda \in \mathcal{C}$. (The action is well defined on quotients because either an orbit is mapped injectively onto a subspace of $\TS/\Lambda$ or it is mapped to the point $[\Lambda] \in \TS/\Lambda$.) So the inclusion and quotient diagrams both admit commuting $G$-actions.
		\end{remark}
		
		If $D$ and $E$ are diagrams of groups, we say a collection of homomorphisms $f = \{f_i \colon G \to H \mid G\in D, H \in E\}$ is a \emph{map of diagrams} if the diagram $D \sqcup_f E$ commutes, where the objects of $D \sqcup_f E$ are given by the disjoint union of the objects in $D$ and $E$ and the homomorphisms of $D \sqcup_f E$ are given by the union of the homomorphisms in $D$ and $E$ together with the homomorphisms in $f$.
		
		\begin{defn}
			Let $f \colon \TS \to \TS'$ be an orbit-preserving map. We define the \emph{induced map on inclusion cohomology diagrams} $f^* \colon \check{H}^*(D_{\TS'}) \to \check{H}^*(D_\TS)$ by $$f^* = \{f|_\Lambda^*\colon \check{H}^*(\Lambda') \to \check{H}^*(\Lambda) \mid f(\Lambda) = \Lambda'\}.$$
		\end{defn}
		Lemma \ref{LEM:orbits} tells us that this induced map of diagrams of groups $f^*$ is non-empty (for all non-empty $\TS'$).
		
		\begin{lemma}\label{LEM:induced}
			The induced map on inclusion cohomology diagrams $f^*$ is a map of diagrams of groups.
		\end{lemma}
		\begin{proof}
			Suppose $f|_\Lambda^*, f|_{\Lambda'}^* \in f^*$ and suppose without loss of generality that $\Lambda \subset \Lambda'$ with inclusion map $i \colon \Lambda \to \Lambda'$. Then we must have $f(\Lambda)\subset f(\Lambda')$ and an inclusion map $j\colon f(\Lambda) \to f(\Lambda')$. If $x \in \Lambda$ then $f|_{\Lambda'} (i (x)) = f|_{\Lambda'}(x) = f(x)$ and $j ( f|_\Lambda (x)) = j(f(x)) = f(x)$. So,
			$$f|_{\Lambda'} \circ i = j \circ f|_{\Lambda}$$ and then by applying the \Cech cohomology functor we get $$i^* \circ f|_{\Lambda'}^* = f|_\Lambda^* \circ j^*$$ as required.
		\end{proof}
		
		For a CIS $\Lambda$ of $\TS$, let $q_\Lambda \colon \TS \to \TS/\Lambda$ be the corresponding quotient map. For an orbit-preserving map $g \colon \TS \to \TS'$, if $g(\Lambda) = \Lambda'$, then there is a unique map $g_\Lambda \colon \TS/\Lambda \to \TS'/\Lambda'$ such that $$g_\Lambda \circ q_\Lambda = q_{\Lambda'} \circ g$$
		\begin{defn}
			Let $g \colon \TS \to \TS'$ be an orbit-preserving map. We define the \emph{induced map on quotient cohomology diagrams} $g^* \colon \check{H}^*(D^{\TS'}) \to \check{H}^*(D^\TS)$ by $$g^* = \{g_\Lambda^* \colon \check{H}^*(\TS'/\Lambda') \to \check{H}^*(\TS/\Lambda) \mid g(\Lambda) = \Lambda'\}$$
		\end{defn}
		Lemma \ref{LEM:orbits} tells us that this induced map $g^*$ is non-empty.
		
		\begin{lemma}
			The induced map on quotient cohomology diagrams $g^*$ is a map of diagrams of groups.
		\end{lemma}
		The proof is very similar to the proof of Lemma \ref{LEM:induced}
		
		\begin{thm}
			Inclusion and quotient cohomology diagrams, together with their induced maps, are contravariant functors from the category of $G$-actions on compact metric spaces and orbit-preserving maps to the category of diagrams of abelian groups and homomorphisms.
		\end{thm}
		
		\begin{proof}
			Let $\TS \stackrel{f}{\to} \TS' \stackrel{g} \to \TS''$ be a pair of orbit preserving maps. Let $\Lambda$ be a CIS of $\TS$. The map of diagrams of groups $f^* \circ g^*$ is the set of all compositions $f|_{\Lambda}^* \circ g|_{f(\Lambda)}^*$ which by functoriality of cohomology is equal to $(g \circ f)|_{\Lambda}^*$. The map of diagrams of groups $(g \circ f)^*$ is the set of all maps $(g \circ f)|_{\Lambda}^*$ for CISs $\Lambda$ of $\TS$ and so $f^* \circ g^* = (g \circ f)^*$ as required.
			
			A similar argument shows the functoriality of the quotient cohomology diagram.
		\end{proof}
		
		\begin{cor}
			Both $\check{H}^*(D_{\TS})$ and $\check{H}^*(D^{\TS})$ are at least as strong an invariant of tiling spaces (up to orbit-equivalence) as \Cech cohomology.
		\end{cor}
		
		We will see in the next section that examples exist where $\check{H}^*(D_{\TS})$ and $\check{H}^*(D^{\TS})$ can distinguish pairs of spaces whose cohomology coincides. So they are in fact strictly stronger invariants than \Cech cohomology on its own.
	
	\subsection{Invariant Subspaces of Substitution Tiling Spaces}
		Let $\sub$ be a tame recognisable substitution, let $\TS$ be the associated tiling space and let $\rho \colon \TS \times \R \to \TS$ be the associated flow on $\TS$ given by $$\rho((w,t),\tau) = (\sigma^{\lfloor t+\tau \rfloor} (w),t+\tau \mod 1).$$ Note that orbits in this setting are precisely the path components of the tiling space. So, even though the previous machinery has been defined for dynamical systems, for tiling spaces the dynamical and topological setting coincide. We could have just as easily considered the set of closed unions of path components, rather than closed invariant subspaces.

		\begin{lemma}\label{LEM:finite-CISs}
			Let $\mathcal{C}$ be the set of CISs for a tame recognisable substitution $\sub$ on the alphabet $\Al$. The set $\mathcal{C}$ is finite.
		\end{lemma}
		To reduce notation, we identify without further comment the tilings $T \in \TS_{\sub_N}$ and $f_{N,0}(T)\in \TS_\sub$ where $f_{N,0}$ is the induced forgetful map which removes collaring information on a collared letter $a_v \in \Al_N$.\
		\begin{proof}
			Let $f_N \colon \Gamma_N \to \Gamma_N$ be the induced substitution map on the $N$-collared AP-complex for $\sub$ where $N$ is one greater than the longest bounded legal word for $\sub$, which is well defined by the tameness of $\sub$. Let $\Lambda \in \mathcal{C}$ be a CIS of $\TS$. As $\Lambda$ is invariant under translation $\rho$, the image of $\Lambda$ under the quotient map $p_N \colon \TS \to \Gamma_N$ must be a subcomplex of $\Gamma_N$.
			
			Now, suppose $\Lambda' \in \mathcal{C}$ and that $p_N(\Lambda) = p_N(\Lambda')$. We want to show that $\Lambda$ and $\Lambda'$ must be the same subspace. Suppose for contradiction and without loss of generality, that $\Lambda'\setminus \Lambda$ is non-empty. Let $T$ be a tiling found in $\Lambda'$ but not $\Lambda$. By construction then, $T$ contains a patch of tiles labelled by the word $u \in \Al^*$ which does not appear in any tiling in $\Lambda$. Suppose $u$ contains an expanding letter (if not, extend $u$ in $T$ until it does contain an expanding letter by tameness).
			
			 Let $d$ be the length of the shortest legal word $v \in \hat{\mathcal{L}}_{\sub}$ such that $u$ is a subword of $\sub^i(v)$ ($d$ may be greater than $1$ as $\sub$ is not necessarily minimal). As $u$ contains an expanding letter, so then must every such $v$. Using recognisability, it is not hard to see that any such $v$ of minimal length is of the form $a_1v', v'a_2$ or $a_1v'a_2$ where $v'$ is a bounded word (possibly empty) and $a_1,a_2$ are expanding letters. As $\sub$ is tame, we conclude that $d \leq N+1$.

			 Of those words $v$ of minimal length such that $u$ is a subword of $\sub^i(v)$ for some $i$, let $n$ be the minimal such power. Let $\tilde{V} = \{v \mid |v|=d, u \subset \sub^n(v)\}$. Finally, let $V$ be the set of legal words of length $2N+1$ which contain a word $v \in \tilde{V}$ as a subword. As $d \leq N+1$, it is certainly true that $d\leq 2N+1$, and so $V$ is non-empty because $\tilde{V}$ is non-empty. Note that $V \subset \Al_N'$. In particular, there is a legal $N$-collared letter $a_v$ for every $v\in V$.
			
			Recall that $\sub_N \colon \TS \to \TS$ is a homeomorphism by recognisability, and this function maps orbits to orbits, and so $\sub_N^{-n}(\Lambda)$ and $\sub_N^{-n}(\Lambda')$ are CISs of $\TS$. By our choice of $T$, $\sub_N^{-n}(T)$ is in $\sub_N^{-n}(\Lambda')$ but not $\sub_N^{-n}(\Lambda)$. The tiling $\sub_N^{-n}(T)$ contains a tile $a_v \in V$ and so there exists a $t \in \R$ so that $T_0 = \sub_N^{-n}(T)-t$ has a tile $a_v \in V$ which contains the origin in its interior.
			As $\Lambda$ and $\Lambda'$ are CISs, $T_0 \in \sub_N^{-n}(\Lambda')$ and $T_0 \notin \sub_N^{-n}(\Lambda)$.
			The image of $T_0$ under the quotient map $p_N$ lies on the interior of the edge of the $N$-collared AP-complex $\Gamma_N$ which is labelled by the $N$-collared letter $a_v$.
			If $p_N(\sub_N^{-n}(\Lambda))$ intersected the interior of this edge, then $\sub_N^{-n}(\Lambda)$ would contain an $N$-collared tiling which contained an $a_v$ tile, but then $\Lambda$ would contain a tiling which contained a patch labelled by the word $u$. This contradicts the choice of $u$ not being a patch in any tiling in $\Lambda$.
			
			It follows that if $p_N(\Lambda) = p_N(\Lambda')$ for CISs $\Lambda, \Lambda'$, then $\Lambda = \Lambda'$. Hence, a CIS is fully determined by the associated subcomplex of the $N$-collared AP-complex to which it is sent under the quotient map. There are only finitely many subcomplexes of any AP-complex and so there can only be finitely many CISs in $\mathcal{C}$ of $\TS$.
		\end{proof}
Lemma \ref{LEM:finite-CISs} has a similar flavour to a result of Durand that says that a linearly recurrent subshift has a finite number of non-periodic subshift factors \cite{D:lin-recurrent}.  

		\begin{remark}
		It is important to note that the choice of $N$ large enough is key in the proof of the above Lemma. If $N$ is not chosen large enough, then the quotient map $p_N \colon \TS \to \Gamma_N$ may send distinct CISs to the same subcomplex of $\Gamma_N$.
		
		As an example, consider the substitution $\sub \colon a \mapsto aba,\: b \mapsto bbab,\: c \mapsto aa$ whose tiling space has exactly one non-empty proper CIS corresponding to the tilings which do not contain the patch labelled by the word $aa$. The $0$-collared AP-complex is `too small' to distinguish this CIS from the entire space $\TS$ in the way described in the above proof.
		
		\end{remark}

The following theorem gives a homeomorphism between a CIS and the inverse limit of a subdiagram of $\Gamma_N$.  
		\begin{thm}\label{THM:subcomplex}
			Let $\sub$ be a tame recognisable substitution.  Let $f_N \colon \Gamma_N \to \Gamma_N$ be the induced substitution map on the $N$-collared AP-complex for $\sub$. There exists an integer $n$ so that for all $\Lambda \in \mathcal{C}$, there exists a subcomplex $\Gamma_{\Lambda} \subset \Gamma_N$ such that  $f_N^{n}(\Gamma_{\Lambda}) = \Gamma_{\Lambda}$ and $\varprojlim(\Gamma_{\Lambda},f_N^n) = \Lambda$.
		\end{thm}		
		\begin{proof}
			As $\sub$ is recognisable, the substitution acts as a homeomorphism on $\TS$ and so the substitution permutes CISs of the tiling space. By Lemma \ref{LEM:finite-CISs}, $\mathcal{C}$ is finite. As such, an integer $n$ can be chosen so that $\sub^n(\Lambda) = \Lambda$ for all $\Lambda \in \mathcal{C}$.
		
			Let $p_N \colon\TS \to \Gamma_N$ be the quotient map from the tiling space to the $N$-collared AP-complex. Let $p = p_N|_\Lambda$, be the restriction of the quotient map to $\Lambda$. As $\Lambda$ is a CIS, the image $\Gamma_{\Lambda}$ of $p$ is a subcomplex of $\Gamma_N$. Recall that $p_N \circ \sub = f_N\circ p_N$, and so
			\begin{equation}\label{EQN:commute}
				p \circ \sub^n = f_N^n \circ p.
			\end{equation}
			
			Let $h_\Lambda \colon \Lambda \to \varprojlim(\Gamma_{\Lambda},f_N^n)$ be defined by $$h_\Lambda(x) = (p(x), p(\sub^{-n}(x)), p(\sub^{-2n}(x)), \ldots )$$ which is well defined by \ref{EQN:commute}. As $h_\Lambda$ is a telescoped version of $h$ with modified domain and codomain, it is clearly injective, so it only remains to show that $h_\Lambda$ is surjective onto the inverse limit.
			
			A point in the inverse limit corresponds to a unique tiling in the tiling space as $\sub_N$ forces the border. Suppose $(x_0, x_1, x_2,\ldots) \in \varprojlim(\Gamma_\Lambda,f_N^n)$ was not in the image of $h_\Lambda$, then there exists some $i$ for which the patch described by the finite subsequence of points $(x_0 ,x_1,\ldots, x_i)$ does not appear in a tiling in $\Lambda$. But this means that the shifted sequence $(x_i, x_{i+1},\ldots)$ is also not in the image of $h_\Lambda$, as the shift is a homeomorphism, and so the point $x_i\in \Gamma_\Lambda$ must not describe the label $w_0$ of the tile at the origin of any tiling in $\Lambda$. This is impossible by how the Anderson-Putnam complex and the quotient map $p$ are defined, as $\Gamma_\Lambda$ is the image of $\Lambda$ under $p$. It follows that no such point in the inverse limit exists and $h_\Lambda$ is surjective.
		\end{proof}

Just as each CIS is homeomorphic to the inverse limit of a subdiagram of $\Gamma_N$, each quotient of $\TS$ by a CIS is homeomorphic to the inverse limit of a quotient of $\Gamma_N$ by a subdiagram.  
The proof of this fact uses the following lemma.  

                \begin{lemma}\label{LEM:quotient-surjective}
                Let $(X_i,\sub_i)$ and $(Y_i,\psi_i)$ be two inverse systems of compact Hausdorff spaces $X_i, Y_i$, and let $X$ and $Y$ denote the respective inverse limits.  
                Let $p_i \colon X \to X_i$ and $q_i \colon Y \to Y_i$ be the projection maps onto approximants.
                Let $s_i \colon X_i \to Y_i$ be a sequence of continuous maps such that $s_i \circ \sub_i = \psi_i \circ s_{i+1}$, and let $s \colon X \to Y$ be the unique continuous map such that $s_i \circ p_i = q_i \circ s$ for every $i\geq 0$.
                If $\sub_i$ and $s_i$ are surjective for every $i\geq 0$ then $s$ is a surjection.
                \end{lemma}
                
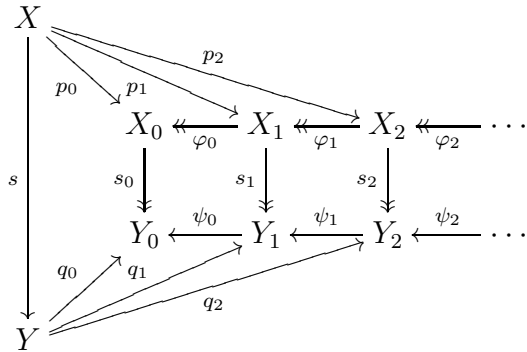
\begin{proof}

\centerline{
\xymatrix{
X \ar@{>}_s[ddd] \ar@{>}_{p_0}[dr] \ar@{>}_{p_1}[drr] \ar@{>}^{p_2}[drrr] & & & & \\
 & X_0 \ar@{>>}_{s_0}[d] & X_1 \ar@{>>}^{\varphi_0}[l] \ar@{>>}_{s_1}[d] & X_2 \ar@{>>}^{\varphi_1}[l] \ar@{>>}_{s_2}[d] & \cdots  \ar@{>>}^{\varphi_2}[l] \\
 & Y_0  & Y_1 \ar@{>}_{\psi_0}[l] & Y_2 \ar@{>}_{\psi_1}[l] & \cdots  \ar@{>}_{\psi_2}[l] \\
Y \ar@{>}^{q_0}[ur] \ar@{>}^{q_1}[urr] \ar@{>}_{q_2}[urrr] %
& & & & 
}
}

                Let $y \in Y$.
                As $Y$ is Hausdorff and compact, $\{y\}$ is closed and hence compact.
                It follows that $q_i(\{y\})$ is compact and hence closed. As $s_i$ is continuous, the preimage $D_i = s_i^{-1}(q_i(\{y\}))$ is then closed, hence compact, and as $s_i$ is surjective, $D_i \neq \emptyset$.
                Note that $\psi_i(q_i(\{y\})) = q_{i-1}(\{y\})$ and so $\psi_i(s_i(D_i)) = q_{i-1}(\{y\})$.
                By commutativity then, $s_{i-1}(\sub_i (D_i)) = q_{i-1}(\{y\})$. This means that $\sub_i D_i$ is a compact subset of $D_{i-1}$ and so by continuity and commutativity, $p_i^{-1}(D_i)$ is a compact subset of $p_{i-1}^{-1}D_{i-1}$.
                Further, each $p_i^{-1} D_i$ is non-empty by the surjectivity of $\sub_i$ for each $i$.
                It follows that $p_0^{-1}D_0 \supset p_1^{-1}D_1 \supset p_2^{-1}D_2 \supset \cdots$ is a nested sequence of closed non-empty subsets of the compact space $X$.
                By Cantor's intersection theorem, $X_0 = \bigcap_{i\geq 0} p_i^{-1} D_i$ is non-empty, and by construction for every $x\in X_0$, $s(x)=y$.
                \end{proof}

                \begin{thm}\label{THM:quotient}
			Let $\sub$ be a tame recognisable substitution.  Let $f_N \colon \Gamma_N \to \Gamma_N$ be the induced substitution map on the $N$-collared AP-complex for $\sub$, let $\Lambda \in \mathcal{C}$, and let $\Gamma_\Lambda$ denote the subcomplex of $\Gamma$ that corresponds to $\Lambda$.  Then for sufficiently large $N$ and for some $n\in\N$, $\TS/\Lambda$ is homeomorphic to $\varprojlim(\Gamma_N/\Gamma_{\Lambda},f_N^n)$.  
                \end{thm}

                \begin{proof}
\centerline{
\xymatrix{
\Omega \ar@{>}^{q_\Lambda}[r] \ar@{>}_{p_N}[d] & \Omega/\Lambda \ar@{>}[d] \\
\Gamma_N \ar@{>}_{Q_\Lambda}[r] & \Gamma_N/\Gamma_\Lambda \\
}
}

Consider the diagram above, where the map in the bottom row is the canonical quotient map from $\Gamma_N$ to $\Gamma_N/\Lambda$, and the map in the right column is the unique continuous map making the diagram commute.  
Let $n$ be as given in Theorem \ref{THM:subcomplex} and $F_N$ be the self-map on $\Gamma_N/\Lambda$ that makes the diagram

\centerline{
\xymatrix{
\Gamma_N \ar@{>}[r] \ar@{>}_{f_N^n}[d] & \Gamma_N/\Gamma_\Lambda \ar@{>}^{F_N}[d] \\
\Gamma_N \ar@{>}[r]  & \Gamma_N/\Gamma_\Lambda  \\
}
}
commute.  

Then the universal property of the inverse limit yields a diagram 

\centerline{
\xymatrix{\label{DIAG:quotient-diagram}
\Omega \ar@{>}^{q_\Lambda}[r] \ar@{>}_{h}[d] & \Omega/\Lambda \ar@{>}^H[d] \\
\varinjlim (\Gamma_N,f_N^n) \ar@{>}_s[r] & \varinjlim (\Gamma_N/\Gamma_\Lambda, F_N). \\
}
}

The map $s$ in the bottom row of this diagram is surjective by Lemma \ref{LEM:quotient-surjective}, and, as $h$ is a homeomorphism, and hence surjective, $s\circ h$ is surjective, which implies that the map $H$ in the right column is surjective as well.  

$H$ is also injective: pick two points $y_1, y_2\in \Omega/\Lambda$ with the same image.  
If neither point is $\Lambda$, then their $q_\Lambda$-preimages $x_1, x_2 \in\Omega$ are distinct points not in $\Lambda$.  
Thus $h(x_1)$ and $h(x_2)$ are sequences that differ beyond some finite index $i$.  
Moreover, there is a finite index $j$ beyond which neither sequence has entries in $\Gamma_\Lambda$.  
Then the images in $\varinjlim (\Lambda_N/\Lambda_\Gamma,F_N)$ of $y_1$ and $y_2$ are $s(h(x_1)), s(h(x_2))$, the entries of which differ beyond index $\max \{ i, j\}$.  

If $y_1 = \Lambda$, then a similar argument shows that the image of $y_2$ is a sequence, the entries of which are not $\Gamma_\Lambda$ beyond a certain index, while the image of $y_1$ has entry $\Lambda$ at every index, and so these images are different.  

Thus $H$ is a continuous bijection.  
Its domain is a quotient of a compact space, and hence is compact.  
$\Gamma_N/\Gamma_\Lambda$ is Hausdorff, as it is a quotient obtained from a compact Hausdorff space by collapsing a compact subspace to a point.  
Thus the codomain of $H$ is an inverse limit of Hausdorff spaces, and hence is Hausdorff.  
Then $H$ is a continuous bijection from a compact space to a Hausdorff space, and hence is a homeomorphism.  
                \end{proof}

		\subsection{Identifying Closed Invariant Subspaces}
		
		Let $\sub$ be a tame recognisable substitution on $\Al$, let $K$ be a subcomplex of $\Gamma_N$ and let $EV(K) = \bigcup_{i \geq 0} (f_N^n)^i(K)$ be the \emph{eventual range} of $K$. The eventual range of a subcomplex is itself a subcomplex. The set of eventual ranges $EV = \{EV(K) \mid K \mbox{ is a subcomplex of } \Gamma_N\}$ therefore forms a finite set.  

Every CIS in $\mathcal{C}$ corresponds to a unique subcomplex in $EV$ given by the image of the CIS under the quotient map $p_N$ to the $N$-collared AP-complex, so $|\mathcal{C}| \leq |EV|$.  
This inequality will often be strict: consider as an example the Chacon substitution $\sub\colon a\mapsto aaba, b\mapsto b$, which is minimal, so there is only one non-trivial CIS, yet for any $n$, the $n$-collared AP-complex will have an element of $EV$ consisting of a single edge that corresponds to a collared $b$-tile.

One can also be in the situation where an eventual range contains multiple expanding edges, yet does not correspond to a CIS: consider the augmented Fibonacci with a handle substitution $\sub \colon a \mapsto aab, b \mapsto ab, c \mapsto c, d \mapsto bca$, where $B = \{c\}$ and so $N=2$ and we have the subcomplex comprising the edges
$$\Gamma = \{a_{abaab}, a_{ababa}, a_{ababc}, a_{baaba}, a_{bcaab}, a_{caaba}, b_{aabaa}, b_{aabab}, b_{babaa}, b_{babca}\}.$$
This subcomplex is an eventual range as $f_N^n(\Gamma) = \Gamma$ but corresponds to no CIS in $\mathcal{C}$.

One can reduce the search for eventual ranges which correspond to a CIS by noting that whenever $\Lambda$ is a CIS, by virtue of $\Lambda$ being translation-invariant, the image of $\Lambda$ under $p_N$ must be a subcomplex which has no leaves (a leaf is a vertex with degree exactly one). It is not immediately clear whether this gives a sufficient condition for identifying all subcomplexes of $\Gamma_N$ which correspond to a CIS in $\mathcal{C}$.

\begin{question}
For a tame recognisable substitution, is there a one-to-one correspondence between the set of leafless eventual ranges of $\Gamma_N$ and the set of CISs $\mathcal{C}$? If not, what condition on a subcomplex $\Gamma \subset \Gamma_N$ is sufficient for $\Gamma$ to correspond to a CIS?
\end{question}
		\section{Examples}\label{SEC:non-min-examples}

		\begin{example}\label{ex:fib+1}
		We define the \emph{Fibonacci substitution with one handle} to be given by $$\sub\colon 0\mapsto 001, \:1\mapsto 01, \:2\mapsto 021$$
		
		By substituting $1$-collared letters (and noting that $B=\emptyset$), we find that there are two non-empty invariant subcomplexes $\Gamma_{\Lambda_1}$ and $\Gamma_{\Lambda_2}$, both fixed under $\sub_1$, corresponding to the collections of $1$-collared letters
		$$\Gamma_{\Lambda_1} = \cup\{[0_{001}], [1_{010}], [0_{100}], [0_{101}]\}$$
		and
		$$\Gamma_{\Lambda_2} = \cup\{[0_{001}], [1_{010}], [2_{021}], [0_{100}], [0_{101}], [0_{102}], [1_{210}]\}.$$
		The $1$-collared AP-complex appears in Figure \ref{fig:fib+1}. An oriented edge from $ab$ to $bc$ denotes an edge labelled by the the letter $b_{abc}$ in the alphabet $\Al_1$ of $1$-collared letters.
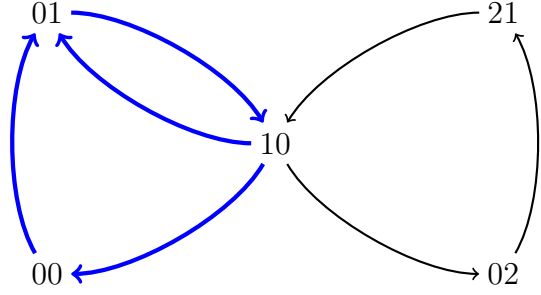
\begin{figure}[H]	
 $$
  \begin{tikzpicture}[->, node distance=3cm, auto, text=black] 

\node [circle,inner sep = 0pt, outer sep = 2pt, minimum size=2mm] (10j) at (0.000000,0.000000) {$10$};
\node [circle,inner sep = 0pt, outer sep = 2pt, minimum size=2mm] (01j) at (-2.999316,1.732446) {$01$};
\node [circle,inner sep = 0pt, outer sep = 2pt, minimum size=2mm] (00j) at (-2.999316,-1.732446) {$00$};

\draw [->,ultra thick, bend left, draw=blue, looseness=0.7] (01j) to (10j);
\draw [->,ultra thick, bend left, draw=blue, looseness=0.7] (10j) to (00j);
\draw [->,ultra thick, bend left, draw=blue, looseness=0.7] (10j) to (01j);
\draw [->,ultra thick, bend left, draw=blue, looseness=0.7] (00j) to (01j);

\node [circle,inner sep = 0pt, outer sep = 2pt, minimum size=2mm] (21k) at (2.999316,1.732446) {$21$};
\node [circle,inner sep = 0pt, outer sep = 2pt, minimum size=2mm] (02k) at (2.999316,-1.732446) {$02$};

\draw [->,thick, bend right, looseness=0.7] (10j) to (02k);
\draw [->,thick, bend right, looseness=0.7] (02k) to (21k);
\draw [->,thick, bend right, looseness=0.7] (21k) to (10j);


\end{tikzpicture}
	$$
	\caption{The $1$-collared AP-complex for the Fibonacci substitution with one handle, with the subcomplex $\Gamma_{\Lambda_1}$ coloured blue.}
	\label{fig:fib+1}
	\end{figure}

		The subcomplex $\Gamma_{\Lambda_1}$ in blue corresponds to a CIS given by considering the restriction of the substitution to the subalphabet $\{0,1\}$ which is (a re-encoding of) the Fibonacci substitution which is connected and has first cohomology $\check{H}^1(\TS_{Fib}) \cong \Z^2$. The subcomplex $\Gamma_{\Lambda_2}$ corresponds to the CIS which is the entire tiling space, which is connected and has first cohomology $\check{H}^1(\TS) \cong \varinjlim\left(\mathbb{Z}^3,\left(\begin{smallmatrix}1&1&0\\1&2&0\\1&1&1\end{smallmatrix}\right)\right) \cong \Z^3$, where the unimodular matrix $\left(\begin{smallmatrix}1&1&0\\1&2&0\\1&1&1\end{smallmatrix}\right)$ is found by choosing appropriate generators of $H^1(\Gamma_1)$. The only other CIS is the empty set.
		
		So $\check{H}^\ast(D_{\TS})$ is given by the diagrams
		$$\begin{array}{rl}
		\check{H}^0(D_{\TS}) \colon & \Z \to \Z \to 0 \\
		\check{H}^1(D_{\TS}) \colon & \Z^3 \to \Z^2 \to 0
		\end{array}$$
		We can use Theorem \ref{THM:quotient} to see, as $\Gamma_1/\Gamma_{\Lambda_1}$ is a circle and $\sub_1$ acts on this quotient complex by a map which is homotopic to the identity, that $\check{H}^i(\TS/\Lambda_1) \cong \varinjlim(H^i(S^1),\operatorname{Id}) \cong \mathbb{Z}$ for $i=0,1$.
		
		So $\check{H}^\ast(D^{\TS})$ is given by the diagrams
		$$\begin{array}{rl}
		\check{H}^0(D^{\TS}) \colon & \Z \to \Z \to \Z \\
		\check{H}^1(D^{\TS}) \colon & 0 \to \Z \to \Z^3
		\end{array}$$
		
		Alternatively, we could have used the fact that $\Lambda_1$ is a closed connected subspace of $\TS$ and so we get an exact sequence in reduced \Cech cohomology
		
		$$0 \to \check{H}^1(\TS/\Lambda) \to \check{H}^1(\TS) \to \check{H}^1(\Lambda) \to 0$$
		which splits (as $\check{H}^1(\Lambda) \cong \Z^2$) to give $\check{H}^1(\TS) \cong \check{H}^1(\TS/\Lambda) \oplus \Z^2$. As above, we can identify $\check{H}^1(\TS/\Lambda)$ with $H^1(S^1)$ and so $\check{H}^1(\TS) \cong \Z^3$.
		
		This distinguishes $\TS$ from the tiling space associated to the Tribonacci substitution which has $\check{H}^0(\TS_{\text{Trib}}) \cong \Z$ and $\check{H}^1(\TS_{\text{Trib}}) \cong \Z^3$ but no proper, non-empty CISs. So the diagrams $\check{H}^\ast(D_{\TS_{\text{Trib}}})$ and $\check{H}^\ast(D^{\TS_{\text{Trib}}})$ have a different shape and so cannot be isomorphic to the diagrams for $\sub$.

See Example \ref{ex:fib-proximal} for an example of a substitution with the same inclusion and quotient cohomology diagrams that nevertheless gives rise to a different tiling space.  
		\end{example}

		Consider the following two substitutions.

		`Two Tribonaccis with a bridge':
		$$\sub_1 \colon 0\mapsto 0201 , \:1\mapsto 001, \:2\mapsto 0, \:\overline{0}\mapsto \overline{0}\overline{2}\overline{0}\overline{1}, \: \overline{1}\mapsto \overline{0}\overline{0}\overline{1}, \:\overline{2}\mapsto \overline{0}, \: X\mapsto 1\overline{0}$$
		`Quadibonacci and Fibonacci with a bridge':
	$$\sub_2 \colon 0\mapsto 0201 , \:1\mapsto 0301, \:2\mapsto 001, \:3\mapsto 0, \: \overline{0}\mapsto \overline{0}\overline{0}\overline{1}, \:\overline{1}\mapsto \overline{0}\overline{1}, \: X\mapsto 1\overline{0}$$
	
	\begin{prop}\label{PROP:trib+trib}
	$\check{H}^\ast(\TS_{\sub_1})$ is isomorphic to $\check{H}^\ast(\TS_{\sub_2})$ but they have degree $1$ inclusion cohomology diagrams
	
		$$
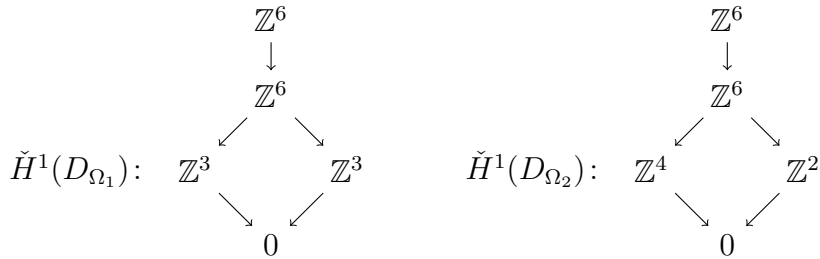
\begin{tikzpicture}[node distance=1cm, auto]
  		\node (1) {};
  		\node (Y) [below of=1] {};
  		\node (X) [right of=1] {$\Z^6$};
  		\node (2) [below of=X] {$\Z^6$};
  		\node (3) [right of=2] {};
  		\node (4) [below of=Y] {$\Z^3$};
  		\node (5) [right of=4] {};
  		\node (6) [right of=5] {$\Z^3$};
  		\node (7) [below of=4] {};
  		\node (8) [right of=7] {$0$};
  		\node (9) [right of=8]{};
  		\draw[->] (2) to node [swap] {} (4);
  		\draw[->] (2) to node [swap] {} (6);
  		\draw[->] (4) to node [swap] {} (8);
  		\draw[->] (6) to node [swap] {} (8);
  		\draw[->] (X) to node [swap] {} (2);
  		
  		\node (A) [left of=4, node distance=1.5cm] {$\check{H}^1(D_{\TS_{1}})\colon$};
  		
  		\node (1b) [right of=1, node distance=6cm]{};
  		\node (Yb) [below of=1b] {};
  		\node (Xb) [right of=1b] {$\Z^6$};
  		\node (2b) [below of=Xb] {$\Z^6$};
  		\node (3b) [right of=2b] {};
  		\node (4b) [below of=Yb] {$\Z^4$};
  		\node (5b) [right of=4b] {};
  		\node (6b) [right of=5b] {$\Z^2$};
  		\node (7b) [below of=4b] {};
  		\node (8b) [right of=7b] {$0$};
  		\node (9b) [right of=8b]{};
  		\draw[->] (2b) to node [swap] {} (4b);
  		\draw[->] (2b) to node [swap] {} (6b);
  		\draw[->] (4b) to node [swap] {} (8b);
  		\draw[->] (6b) to node [swap] {} (8b); 
  		\draw[->] (Xb) to node [swap] {} (2b);
  		
  		\node (Ab) [left of=4b, node distance=1.5cm] {$\check{H}^1(D_{\TS_{2}})\colon$};
		\end{tikzpicture}$$
		
		and degree $1$ quotient cohomology diagrams
	
	$$
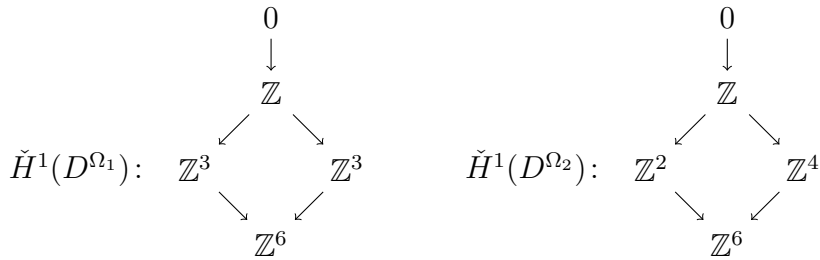
\begin{tikzpicture}[node distance=1cm, auto]
  		\node (1) {};
  		\node (Y) [below of=1] {};
  		\node (X) [right of=1] {$0$};
  		\node (2) [below of=X] {$\Z$};
  		\node (3) [right of=2] {};
  		\node (4) [below of=Y] {$\Z^3$};
  		\node (5) [right of=4] {};
  		\node (6) [right of=5] {$\Z^3$};
  		\node (7) [below of=4] {};
  		\node (8) [right of=7] {$\Z^6$};
  		\node (9) [right of=8]{};
  		\draw[->] (2) to node [swap] {} (4);
  		\draw[->] (2) to node [swap] {} (6);
  		\draw[->] (4) to node [swap] {} (8);
  		\draw[->] (6) to node [swap] {} (8);
  		\draw[->] (X) to node [swap] {} (2);
  		
  		\node (A) [left of=4, node distance=1.5cm] {$\check{H}^1(D^{\TS_{1}})\colon$};
  		
  		\node (1b) [right of=1, node distance=6cm]{};
  		\node (Yb) [below of=1b] {};
  		\node (Xb) [right of=1b] {$0$};
  		\node (2b) [below of=Xb] {$\Z$};
  		\node (3b) [right of=2b] {};
  		\node (4b) [below of=Yb] {$\Z^2$};
  		\node (5b) [right of=4b] {};
  		\node (6b) [right of=5b] {$\Z^4$};
  		\node (7b) [below of=4b] {};
  		\node (8b) [right of=7b] {$\Z^6$};
  		\node (9b) [right of=8b]{};
  		\draw[->] (2b) to node [swap] {} (4b);
  		\draw[->] (2b) to node [swap] {} (6b);
  		\draw[->] (4b) to node [swap] {} (8b);
  		\draw[->] (6b) to node [swap] {} (8b);
  		\draw[->] (Xb) to node [swap] {} (2b);
  		
  		\node (Ab) [left of=4b, node distance=1.5cm] {$\check{H}^1(D^{\TS_{2}})\colon$};
		\end{tikzpicture}$$
	\end{prop}

\begin{proof}
The proof is left as an exercise.  
\end{proof}

	Hence, $\check{H}^1(D_{\TS})$ and $\check{H}^1(D^{\TS})$ can distinguish tiling spaces which have the same cohomology and lattice structure of CISs.

%
%
%

		\subsection{Discussion}
			\subsubsection{Barge-Diamond Complexes for Non-primitive Substitutions}
	One may ask why we have been using collared Anderson-Putnam complexes and not Barge-Diamond complexes \cite{BD} in the sections focussing on non-minimal substitutions. Indeed, (a slightly modified version of) the BD-complex is a suitable replacement for the $N$-collared AP-complex, and most results from the previous section would hold with very little changed. However, the advantages afforded to the Barge-Diamond method are less apparent when there exist bounded letters in the alphabet. When all letters are expanding and the substitution is aperiodic, a very similar argument to the original proof presented by Barge and Diamond \cite{BD} will carry through, and one can then apply the usual method of replacing the induced substitution on the BD-complex with a homotopic map which is simplicial on the vertex-edges.
	
	When there exist bounded words in the subshift, the usual BD-complex with an $\epsilon$-ball collaring at each point\footnote{See \cite{BDHS:other} for an explanation of what it means to \emph{collar points} in the tiling, instead of collaring tiles.} does not suffice to get the necessary homeomorphism to the inverse limit (for broadly the same reasons that the $1$-collaring does not suffice to induce border-forcing when $B$ is non-empty).

	Instead, the approach that one should take is to collar points with a ball of radius $N-1+\epsilon$ at each point---this is equivalent to replacing the substitution with its $(N-1)$-collared substitution and then using the $\epsilon$-ball collaring on this collared substitution (and so we are using the usual BD-complex $K_{\sub_{N-1}}$ for the collared substitution $\sub_{N-1}$). This has the advantage of needing to collar out one fewer times than in the AP-complex approach. Moreover, we can still replace the induced substitution map with a homotopic map which acts simplicially on the distinguished subcomplex of transition edges. Unlike in the minimal case, it is not necessarily true that $\tilde{H}^0(K_{\sub_{N-1}})$ is trivial, as $\TS_\sub$ may have multiple connected components and so extension problems coming from the Barge-Diamond exact sequence will in general be more difficult.
	
	To illustrate this alternative method, we present a brief example calculation of cohomology for the Chacon substitution.
	\begin{example}
	Let $\sub$ be given by $\sub \colon a \mapsto aaba,\: b \mapsto b$, the Chacon substitution on the alphabet $\{a, b\}$. Let $$1 = a_{aaa},\: 2 = a_{aab},\: 3 = b_{aba},\: 4 = a_{bab},\: 5 = a_{baa}.$$ The $1$-collared substitution is given by $$\sub_1 \colon 1 \mapsto 1235,\: 2 \mapsto 1234,\: 3 \mapsto 3,\: 4 \mapsto 5234,:\ 5 \mapsto 5325$$ and the BD-complex is given in Figure \ref{fig:chacon-bd}.
	\begin{figure}[h]
	 $$\begin{tikzpicture}[->, node distance=2cm, auto] 
\node [fill,circle,draw,inner sep = 0pt, outer sep = 0pt, minimum size=1mm] (ai) at (2.000000,0.000000) {};
\node [fill,circle,draw,inner sep = 0pt, outer sep = 0pt, minimum size=1mm] (ao) at (1.618173,1.175379) {};
\node [fill,circle,draw,inner sep = 0pt, outer sep = 0pt, minimum size=1mm] (bi) at (0.618485,1.901966) {};
\node [fill,circle,draw,inner sep = 0pt, outer sep = 0pt, minimum size=1mm] (bo) at (-0.617358,1.902333) {};
\node [fill,circle,draw,inner sep = 0pt, outer sep = 0pt, minimum size=1mm] (ci) at (-1.617476,1.176338) {};
\node [fill,circle,draw,inner sep = 0pt, outer sep = 0pt, minimum size=1mm] (co) at (-2.000000,0.001185) {};
\node [fill,circle,draw,inner sep = 0pt, outer sep = 0pt, minimum size=1mm] (di) at (-1.618870,-1.174419) {};
\node [fill,circle,draw,inner sep = 0pt, outer sep = 0pt, minimum size=1mm] (do) at (-0.619612,-1.901600) {};
\node [fill,circle,draw,inner sep = 0pt, outer sep = 0pt, minimum size=1mm] (ei) at (0.616230,-1.902698) {};
\node [fill,circle,draw,inner sep = 0pt, outer sep = 0pt, minimum size=1mm] (eo) at (1.616779,-1.177296) {};
\draw (ai) edge[bend right=110, looseness=3, ->, thick, swap] node {$1$}(ao);
\draw (bi) edge[bend right=110, looseness=3, ->, thick, swap] node {$2$}(bo);
\draw (ci) edge[bend right=110, looseness=3, ->, thick, swap] node {$3$}(co);
\draw (di) edge[bend right=110, looseness=3, ->, thick, swap] node {$4$}(do);
\draw (ei) edge[bend right=110, looseness=3, ->, thick, swap] node {$5$}(eo);
\draw [-,thick, bend right, draw=white, line width=4pt, looseness=0.7] (ao) to (bi);
\draw [->, thick, bend right, looseness=0.7] (ao) to (bi);
\draw [-,thick, bend right, draw=white, line width=4pt, looseness=0.7] (bo) to (ci);
\draw [->, thick, bend right, looseness=0.7] (bo) to (ci);
\draw [-,thick, bend right, draw=white, line width=4pt, looseness=0.7] (co) to (di);
\draw [->, thick, bend right, looseness=0.7] (co) to (di);
\draw [->, thick, bend right, looseness=0.7, draw=red] (co) to (ei);
\draw [-,thick, bend right, draw=white, line width=7pt, looseness=0.7] (do) to (ci);
\draw [->, thick, bend right, looseness=0.7, draw=red] (do) to (ci);
\draw [-,thick, bend right, draw=white, line width=4pt, looseness=0.7] (eo) to (ai);
\draw [->, thick, bend right, looseness=0.7, draw=red] (eo) to (ai);
\draw [->, thick, bend right, looseness=0.7] (eo) to (bi);
 \end{tikzpicture} $$
 
 \caption{The Barge-Diamond complex $K_{\sub_1}$ for the $1$-collared Chacon substitution with the subcomplex of transition edges in the eventual range coloured red}
 \label{fig:chacon-bd}
 \end{figure}
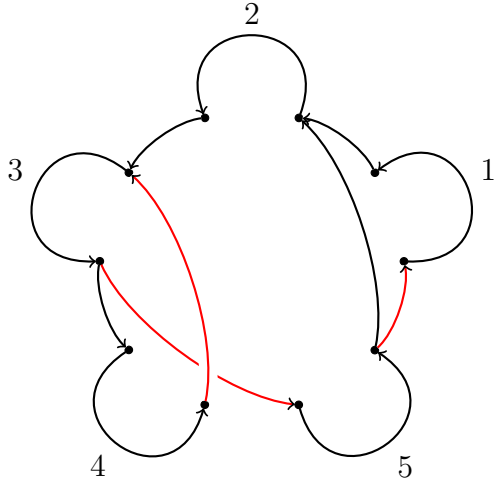
	
	The eventual range of the map $g$ acting on the subcomplex $S$ of transition edges is the collection $\{e_{35}, e_{43}, e_{51}\}$ coloured in red. The substitution acts on this eventual range like the identity. Note that $S$ has exactly three connected components, all of which are contractible. It follows that the Barge-Diamond exact sequence for this substitution is given by $$0 \to \Z^2 \to \varinjlim \left(\Z^5, \left(
	\begin{smallmatrix}
	1&1&1&0&1\\
	1&1&1&1&0\\
	0&0&1&0&0\\
	0&1&1&1&1\\
	0&1&1&0&2
	\end{smallmatrix}
	\right)\right) \to \check{H}^1(\TS) \to 0 \to 0$$
	\end{example}
	Experience with examples seems to suggest that it is often the case that the eventual range of $S$ under the induced substitution will often have multiple connected components whenever $N > 1$ and especially when $\sub$ is not minimal, and so we seem to lose the advantage normally afforded to us with Barge-Diamond calculation where it is often the case that the exact sequence splits. In fact, it is probably more efficient in the above example to directly find generators of the cohomology of the entire complex $K_{\sub_1}$ (where in this case there are only three generators) and to calculate the induced substitution on $H^1(K_{\sub_1})$ in order to calculate $\check{H}^1(\TS)$. If we do that, we find that $\check{H}^1(\TS) \cong \varinjlim \left(\Z^3, \left(\begin{smallmatrix}0&1&0\\-1&3&1\\-1&1&1\end{smallmatrix}\right)\right)$.

	\subsubsection{Extensions of Substitutions by Other Substitutions}
		So far, our only examples of non-minimal substitutions that have been presented have been relatively tame---the tiling spaces have all been a finite collection of minimal tiling spaces which are possibly connected by a finite number of path components which asymptotically approach some sub-collection of the minimal sets. In particular, by quotienting out by the disjoint union of the minimal sets, we are left with a space homeomorphic to a cell complex. While these spaces are interesting, and serve as good test cases for our machinery, the range of possible behaviours for non-minimal substitutions is much more varied.
		
		For instance, we could break the asymptotic behaviour in the above described examples, and instead have new path components which approach minimal sets proximally, instead of asymptotically. 

\begin{example}\label{ex:fib-proximal}
Consider the substitution $$\sub \colon 0 \mapsto 001,\: 1 \mapsto 01,\: 2\mapsto X021X,\: X \mapsto X$$ whose proximal path component is the orbit of the word $$\ldots 0010010100101X00100101X001X021X01X00101X00100101001001 \ldots$$ where the sparse appearances of the symbols $X$, which become more rare the further one travels from the single $2$, serve to break the asymptotic nature of the handle.
There is also a single asymptotic handle associated to the bi-infinite word $$\ldots 0010010100101X00100101001001 \ldots.$$
The lattice of CISs for this substitution is $\emptyset \to \Lambda_{Fib} \to \Lambda_{Fib+1} \to \TS$ where $\Lambda_{Fib}$ is a Fibonacci tiling space, $\Lambda_{Fib+1}$ includes the asymptotic handle and $\TS$ is the full tiling space which includes the proximal handle.

The inclusion cohomology diagram in degree $1$ is given by
$H^1(D_{\TS})\colon \: \Z^4 \to \Z^3 \to \Z^2 \to 0$.
\end{example}

		To support this direction of exploring more varied behaviour, we introduce a curious family of examples where the quotients of the tiling space by the CISs are of particular interest, and where there is a natural factor map onto the minimal set of the tiling space. In particular the complement $\TS \setminus \Lambda_{\operatorname{min}}$ will often have uncountably many path components, where $\Lambda_{\operatorname{min}}$ is the disjoint union of the minimal CISs. One might think of this construction as `extending' one substitution by another in a proximal fashion. 

Suppose that $\sub$ and $\psi$ are substitutions on $\Al$ and $\Be$ respectively, with $\sub$ primitive and suppose that $|\Be| \leq |\Al|$. Let $i \colon \Be \to \Al$ be an injection. Assume that for each $b \in \Be$, if $\psi(b) = b_1 \ldots b_n$, then there exists an interior subsequence $(a_{k_1},\ldots, a_{k_n})$ of $\sub(i(b)) = a_1 \ldots a_m$ of the form $(i(b_1), \ldots, i(b_n))$ (if not, take a high enough power of $\sub$ so that there is). Here by interior, we mean that $a_{k_1} \neq a_1$ and $a_{k_n} \neq a_m$.

\begin{defn}
Let $\sub$ and $\psi$ be as above and choose an injection $i\colon \Be \to \Al$ and a set of subsequences $S = \{s_{b} = (a_{k_1},\ldots, a_{k_n}) \mid b \in \Be\}$ of $\sub(i(b))$ as above.

Define a new substitution $[\sub,\psi]_S$ on the alphabet $\Al \sqcup \Be$ by $[\sub,\psi]_S(a) = \sub(a)$ for all $a \in \Al$ and for $b \in \Be$ by $[\sub,\psi]_S(b) = \sub(i(b))$ except replace the occurrence of $a_{k_j}$ with $b_j$.
\end{defn}
There is a natural factor map $\TS_{[\sub,\psi]_S} \to \TS_\sub$ given by mapping the letters $b \in \Be$ to $i(b)$.

\begin{example}\label{EX:solenoid}
If $\sub \colon 0 \mapsto 00100101,\: 1 \mapsto 00101$ and $\psi \colon a \mapsto aa$, then we could choose the injection $a \mapsto 0$ and then choose as the subsequence of $\sub(i(a)) = 0_{(1)} 0_{(2)} 1_{(3)} 0_{(4)} 0_{(5)} 1_{(6)} 0_{(7)} 1_{(8)}$ the sequence $(0_{(4)}, 0_{(5)})$ so $S = \{(0_{(4)},0_{(5)})\}$. Then our extended substitution $[\sub,\psi]_S$ is given by $$[\sub,\psi]_S \colon 0 \mapsto 00100101,\: 1\mapsto 00101,\: a \mapsto 001aa101.$$
\end{example}

\begin{example}
Let $\psi = \operatorname{Id}$ be the substitution on the alphabet $\{x\}$ given by $\operatorname{Id}(x) = x$, and let $i \colon \{x\} \to \Al$ be given by $i(x)=a$ for some $a \in \Al$. As $\sub$ is primitive by assumption, let $a_{k_1}$ be an occurrence of the letter $a$ in the interior of the word $\sub^n(a)$ for some positive natural $n$. Let $S=\{(a_{k_1})\}$.

The substitution $[\sub,\operatorname{Id}]_S$ is just the substitution $\sub$ with a single handle. That is, the tiling space for $[\sub,\operatorname{Id}]_S$ is just the tiling space for $\sub$ with a single extra one-dimensional path component which asymptotically approaches the minimal component in both directions. The image of the handle under the factor map onto $\TS_\sub$ is precisely the orbit of the limit word $\lim_{j \to \infty}\sub^{jn}(a)$ expanded about the interior letter $a_{k_1}$ appearing in $\sub^n(a)$. By iterating this method, we can add as many handles as we like.
\end{example}

In general, the substitution tiling space $\TS_{[\sub,\psi]_S}$ has exactly one non-empty proper CIS which is exactly the tiling space $\TS_\sub$ given by restriction of the substitution to the subalphabet $\Al$.

There is a close relationship between the quotient complex $\Gamma_N/\Gamma_{\TS_\sub}$ and the AP-complex $\Gamma_\psi$ of the substitution $\psi$. Let $f \colon \Gamma_N/\Gamma_{\TS_\sub} \to \Gamma_N/\Gamma_{\TS_\sub}$ and $g \colon \Gamma_{\psi} \to \Gamma_{\psi}$ be the respective bonding maps. It would appear that more often than not there is a map $h \colon \Gamma_N/\Gamma_{\TS_\sub} \to \Gamma_\psi$ which conjugates these bonding maps up to homotopy, that is $g \circ h \simeq h \circ f$. This would seem to suggest a close relationship between the spaces $\TS_{[\sub,\psi]_S}/\TS_\sub$ and $\TS_\psi$, perhaps up to shape equivalence\footnote{For an introduction and overview of the r\^{o}le of shape theory in the study of tiling spaces, we refer the reader to \cite{CH:codim-one-attractors}}.

\begin{question}
What is the relationship between $\TS_{[\sub,\psi]_S}/\TS_\sub$ and $\TS_\psi$?  
\end{question}

The importance of the choice of the set of subsequences $S$ in the construction of $[\sub,\psi]_S$ is not immediately apparent. It seems unlikely that the resulting tiling space is independent of the choice of $S$. By taking powers of $\sub$, one can generate infinitely many distinct such choices. By construction, the inclusion and quotient cohomology diagrams of these spaces will all be very similar (if not identical), and so a stronger invariant is likely needed to distinguish such substitutions topologically.

\begin{question}
Does there exist a pair of substitution $\sub,\psi$ and sets of subsequences $S, S'$ such that $\TS_{[\sub,\psi]_S}$ and $\TS_{[\sub,\psi]_{S'}}$ are not homeomorphic? If such behaviour is typical, what tools are needed to topologically or dynamically distinguish such pairs of spaces in general?
\end{question}

\begin{example}\label{EX:proximal}
	In \cite{BD:proximal}, Barge and Diamond outline a method for associating, to a primitive aperiodic substitution $\sub$, a new substitution $\tilde{\sub}$ which is non-minimal.  They show that the homeomorphism type of the tiling space $\TS_{\tilde{\sub}}$ is a homeomorphism invariant of the tiling space $\TS_\sub$, and so the cohomology $\check{H}^i(\TS_{\tilde{\sub}})$ is also a topological invariant for $\TS_\sub$. The method for forming the substitution $\tilde{\sub}$ from the so-called \emph{balanced pairs of words associated to pairs of asymptotic composants} is involved, and it would be cumbersome to reproduce the construction here, so the reader is referred to the paper \cite{BD:proximal}.
	
	Using this construction, it can be shown that given the Fibonacci substitution $\sub_{Fib} \colon 0\mapsto 001,\: 1\mapsto 01$, the associated substitution $\tilde{\sub}_{Fib}$ is given by $\tilde{\sub}_{Fib} \colon a \mapsto aab,\: b \mapsto ab,\: c \mapsto acab$. The tiling space of this substitution is orbit equivalent to a Fibonacci with one handle substitution $[\sub_{Fib},\operatorname{Id}]_S$ (the equivalence is given by the single $c$ tile absorbing the $a$ tile to its right).
	\end{example}

	\begin{example}\label{EX:asymptotic}
	Considering the substitutions
	$$
	\begin{array}{rlll}
	\sub_1 \colon & a \mapsto cab & b \mapsto ac & c \mapsto a\\
	\sub_2 \colon & a \mapsto bbac & b \mapsto a & c \mapsto b.
	\end{array}
	$$
	It is an exercise for the reader to check that we have cohomology groups $\check{H}^1(\TS_{\sub_1}) \cong \check{H}^1(\TS_{\sub_2}) \cong \Z^5$. So, cohomology does not distinguish the tiling spaces of these two substitutions. It is also the case that several other invariants of primitive substitution tiling spaces fail to distinguish these substitutions. We can instead form the two new substitutions $\tilde{\sub}_1, \tilde{\sub}_2$. We omit the specific presentations of these substitutions owing to their extremely large size---$\tilde{\sub_1}$ has an alphabet on 19 letters, $\tilde{\sub_2}$ has an alphabet on 87 letters.
	
	Using the results of this work, we can calculate that $\operatorname{rk}\check{H}^1(\TS_{\tilde{\sub}_1}) = 17$ and $68 \leq \operatorname{rk}\check{H}^1(\TS_{\tilde{\sub}_2}) \leq 74$ and so by the result of Barge and Diamond, these invariants distinguish the substitutions $\sub_1$ and $\sub_2$. Hence we have $\TS_{\sub_1} \not\cong \TS_{\sub_2}$.
	
	\emph{Acknowledgement.} The authors thank Scott Balchin for writing a computer program to determine the substitutions $\tilde{\sub}_1$ and $\tilde{\sub}_2$ after it became apparent that hand calculations would not be feasible in a reasonable amount of time.
\end{example}


\bibliographystyle{abbrv}
\bibliography{bib-primitivity2}

\begin{thebibliography}{10}

\bibitem{AP}
J.~E. Anderson and I.~F. Putnam.
\newblock Topological invariants for substitution tilings and their associated
  {$C^*$}-algebras.
\newblock {\em Ergodic Theory Dynam. Systems}, 18(3):509--537, 1998.

\bibitem{BD:proximal}
M.~Barge and B.~Diamond.
\newblock Proximality in {P}isot tiling spaces.
\newblock {\em Fund. Math.}, 194(3):191--238, 2007.

\bibitem{BD}
M.~Barge and B.~Diamond.
\newblock Cohomology in one-dimensional substitution tiling spaces.
\newblock {\em Proc. Amer. Math. Soc.}, 136(6):2183--2191, 2008.

\bibitem{BDH:asymp-orbits}
M.~Barge, B.~Diamond, and C.~Holton.
\newblock Asymptotic orbits of primitive substitutions.
\newblock {\em Theoret. Comput. Sci.}, 301(1-3):439--450, 2003.

\bibitem{BDHS:other}
M.~Barge, B.~Diamond, J.~Hunton, and L.~Sadun.
\newblock Cohomology of substitution tiling spaces.
\newblock {\em Ergodic Theory Dynam. Systems}, 30(6):1607--1627, 2010.

\bibitem{BKM:recognisable}
S.~Bezuglyi, J.~Kwiatkowski, and K.~Medynets.
\newblock Aperiodic substitution systems and their {B}ratteli diagrams.
\newblock {\em Ergodic Theory Dynam. Systems}, 29(1):37--72, 2009.

\bibitem{CH:codim-one-attractors}
A.~Clark and J.~Hunton.
\newblock Tiling spaces, codimension one attractors and shape.
\newblock {\em New York J. Math.}, 18:765--796, 2012.

\bibitem{CS:non-primitive-invariant-measures}
M.~I. Cortez and B.~Solomyak.
\newblock Invariant measures for non-primitive tiling substitutions.
\newblock {\em J. Anal. Math.}, 115:293--342, 2011.

\bibitem{DL:linearly-recurrent}
D.~Damanik and D.~Lenz.
\newblock Substitution dynamical systems: characterization of linear
  repetitivity and applications.
\newblock {\em J. Math. Anal. Appl.}, 321(2):766--780, 2006.

\bibitem{D:main-result}
F.~Durand.
\newblock A characterization of substitutive sequences using return words.
\newblock {\em Discrete Math.}, 179(1-3):89--101, 1998.

\bibitem{D:lin-recurrent}
F.~Durand.
\newblock Linearly recurrent subshifts have a finite number of non-periodic
  subshift factors.
\newblock {\em Ergodic Theory and Dynamical Systems}, 20:1061--1078, 8 2000.

\bibitem{GM:multi-one-d}
F.~G{\"a}hler and G.~R. Maloney.
\newblock Cohomology of one-dimensional mixed substitution tiling spaces.
\newblock {\em Topology Appl.}, 160(5):703--719, 2013.

\bibitem{LM:book}
D.~Lind and B.~Marcus.
\newblock {\em An introduction to symbolic dynamics and coding}.
\newblock Cambridge University Press, Cambridge, 1995.

\bibitem{M:aperiodic}
B.~Moss{\'e}.
\newblock Puissances de mots et reconnaissabilit\'e des points fixes d'une
  substitution.
\newblock {\em Theoret. Comput. Sci.}, 99(2):327--334, 1992.

\bibitem{R:mixed-subs}
D.~Rust.
\newblock An uncountable set of tiling spaces with distinct cohomology.
\newblock {\em Topology and its Applications},
  doi:10.1016/j.topol.2016.01.020:--, 2016.

\bibitem{S:book}
L.~Sadun.
\newblock {\em Topology of tiling spaces}, volume~46 of {\em University Lecture
  Series}.
\newblock American Mathematical Society, Providence, RI, 2008.

\end{thebibliography}

\end{document}